\documentclass{article}

\usepackage{tensor}
\usepackage{graphicx}
\usepackage[utf8]{inputenc}
\usepackage{amsmath, amssymb}
\usepackage[]{graphicx, url}
\usepackage[]{indentfirst}
\usepackage{siunitx} 
\usepackage{authblk}
\usepackage{amsthm}

\newtheorem{theorem}{\bf Theorem}[section]

\newtheorem{remark}{\bf Remark}[section]
\newtheorem{definition}{\bf Definition}[section]

\numberwithin{equation}{section}

\title{Ideal Magnetohydrodynamic Equations on a Sphere and Elliptic-Hyperbolic Property}

\author[1]{Ian Holloway }
\author[2]{Sivaguru S. Sritharan}

\affil[1]{Department of Mathematics, Wright State University, Dayton, OH 45435\\ email: iancholloway@gmail.com}
\affil[2]{M.S. Ramaiah University of Applied Sciences, Bengaluru, India  \\ email: provostsritharan@gmail.com}

\begin{document}
\maketitle

\begin{abstract}
This work contains the derivation and type analysis of the conical Ideal Magnetohydrodynamic equations. The 3D Ideal MHD equations with Powell source terms, subject to the assumption that the solution is conically invariant, are projected onto a unit sphere using tools from tensor calculus. Conical flows provide valuable insight into supersonic and hypersonic flow past bodies, but are simpler to analyze and solve numerically. Previously, work has been done on conical inviscid flows governed by the Euler equations with great success. It is known that some flight regimes involve flows of ionized gases, and thus there is motivation to extend the study of conical flows to the case where the gas is electrically conducting. To the authors' knowledge, the conical magnetohydrodynamic equations have never been derived and so this paper is the first invesitgation of that system. Among the results, we show that conical flows for this case do exist mathematically and that the governing system of partial differential equations is of mixed type. Throughout the domain it can be either hyperbolic or elliptic depending on the solution.
\end{abstract}

\section{Introduction}


In previous works, supersonic flow past a cone of arbitrary cross section was projected onto the sphere under the assumption of conical invariance. Conical flows have been successfully used to study general supersonic flows, but with a simpler mathematical framework. The conical flow assumption provides a dimension reduction to the case of 3D flow by assuming that the flow is uniform in the $r$ coordinate direction rather than in one of the cartesian coordinate directions, $x$, $y$, or $z$. Probably the most prolific of such studies was that done by Taylor and Maccoll \cite{TaylorMaccoll} for right circular cones, which is still used in some applications to study supersonic and even hypersonic flows. Subsequent work has considered more general settings than Taylor and Maccoll, such as Ferri \cite{NASA_con} who considered cones at angle of attack and Sritharan and Guan \cite{Guan,Sri_FV} who considered cones of arbitrary cross section at angle of attack with the assumption of irrotational flow. In these works, the flow was not assumed to be conducting. As the aerospace industry moves more and more into hypersonics, it becomes important to add this assumption to the study of aerodynamics. 

It is known that in some flight regimes, particularly at high altitude, and high velocity a plasma sheath can form around the aircraft \cite{ShangRecentResearch}. This sheath has many electromagnetic properties which are important to study. Furthermore, many proposals have been made about how to use electromagnetic forces in active ways to propel and control various types of aircraft. These range from plasma actuators in place of control surfaces, to conditioning the incoming flow stream at the inlet of a scramjet, to even using solely electromagnetic propulsion \cite{Prospects,ShangRecentResearch,ShangCompMAD}. 

To the knowledge of the authors, conical Ideal Magnetohydrodynamic flows have not yet been investigated. Because the conical assumption has been used with such success and electrically conducting is an important topic for the future of aerodynamic design, we build upon previous work in conical flow by deriving and analyzing the type of the conical Ideal Magnetohydrodynamic equations. This work follows closely that by Sritharan and Holloway \cite{ConEuler} on the conical Euler equations and likewise results in a system of equations which do not reference any particular coordinate system. A numerical method which solves these equations is developed in \cite{ConSolver}, and it can be seen that the coordinate free form has a natural compatibility with structured meshes.

The governing equations derived are given here:

\begin{subequations}\label{TheEq}
\begin{gather}
 \frac{\partial}{\partial \xi^\beta}\left(\rho\sqrt{g}v^\beta\right) + 2\rho \sqrt{g}V^3 = 0 
\label{mass} \\
\left[\sqrt{g}\left(\rho v^\alpha v^\beta - \frac{1}{\mu}b^\alpha b^\beta + g^{\alpha\beta}\left[P + \frac{|\pmb{B}|^2}{2\mu}\right]\right)\right]_{||\beta} + 3\left[\sqrt{g}\left(\rho v^\alpha V^3 - \frac{1}{\mu}b^\alpha B^3 \right)\right] \nonumber \\= -\frac{\sqrt{g}}{\mu}b^\alpha(b^\nu_{||\nu} + 2B^3)
\label{momentum} \\
 \frac{\partial }{\partial \xi^\alpha}\left( \rho\sqrt{g}V^3v^\alpha - \frac{\sqrt{g}}{\mu}B^3b^\alpha \right)+ 2\sqrt{g}\left(\rho (V^3)^2-\frac{1}{\mu}(B^3)^2\right) \nonumber \\ 
 - \rho\sqrt{g}v_c^2 + \frac{\sqrt{g}}{\mu}b_c^2 = -\frac{\sqrt{g}}{\mu}B^3 (b^\nu_{||\nu} + 2B^3)
\label{mom3}\\
 \frac{\partial}{\partial \xi^\beta}\left( \sqrt{g}\left[\left(\rho E+P+\frac{|\pmb{B}|^2}{\mu}\right)v^\beta - \frac{1}{\mu}(\pmb{V}\cdot\pmb{B})b^\beta\right] \right) \nonumber \\ 
 + 2\sqrt{g}\left[\left(\rho E+P+\frac{|\pmb{B}|^2}{\mu}\right)V^3 - \frac{1}{\mu}(\pmb{V}\cdot\pmb{B})B^3\right] \nonumber \\ 
= -\frac{\sqrt{g}}{\mu}(\pmb{V}\cdot\pmb{B})(b^\nu_{||\nu} + 2B^3)
\label{energy} \\
 (v^\beta b^\alpha - v^\alpha b^\beta)_{||\beta} + (V^3 b^\alpha - v^\alpha B^3) = -v^\alpha(b^\nu_{||\nu} + 2B^3) \label{magnet} \\
 \frac{\partial}{\partial \xi^\beta}(v^\beta B^3 - V^3b^\beta) + (v^\beta B^3 - V^3b^\beta)\frac{1}{\sqrt{g}}\frac{\partial\sqrt{g}}{\partial\xi^\beta} = -V^3(b^\nu_{||\nu} + 2B^3) \label{mag3}
\end{gather}
\end{subequations}

Einstein summation is used where there are repeated indices. Greek indices take values from 1 to 2. Variables represent the following: $\rho$ is the density, $v^\beta$ are velocity components on the surface of a sphere scaled to have no ``$r$'' dependency, $V^3$ is the radial component of velocity, $E=e+\frac{1}{2}|\pmb{V}|^2$ is the total specific energy (thermal plus kinetic), where $e$ is the specific thermal energy, $b^\beta$ are the magnetic field components on the surface of a sphere scaled to have no ``$r$'' dependence, $B^3$ is the radial component of the magnetic field, $\mu$ is the permeability constant, $P=P(\rho,e)$ is the thermodynamic pressure, $g_{\beta\alpha}$ is the metric tensor characterizing angle and distance on the surface of the sphere (with ``$r$'' dependency removed) which has inverse $g^{\omega\alpha}$, $g$ is the determinant of the metric tensor, $\xi^\beta$ are the coordinates on the surface of the sphere, $v_c=\sqrt{g_{\alpha\beta}v^\alpha v^\beta}$ and $b_c=\sqrt{g_{\alpha\beta}b^\alpha b^\beta}$ are the magnitudes of the surface components of the velocity and magnetic fields respectively. The notation $(\cdot)_{||\beta}$ refers to the covariant derivative on the surface of the sphere, which will have different forms depending on the type of tensor being differentiated.

Treating $\xi^1$ as time-like, four of the characteristic speeds ($\lambda$) can be expressed explicitly in terms of the solution. They are:

\begin{equation}\label{firstFour}
    \lambda=\frac{v^2}{v^1},\frac{v^2}{v^1},\frac{b^2 \pm \sqrt{\mu\rho}v^2 }{b^1 \pm \sqrt{\mu\rho}v^1}
\end{equation}

The last four speeds satisfy the relationship:

\begin{multline}\label{lastFour}
    (v^2-\lambda v^1)^2 =\\ 
    \frac{g}{2\mu\rho} \Biggl[ (g^{22} - 2g^{12}\lambda + g^{11}\lambda^2) \frac{(b_1)^2 (|\pmb{B}|^2+c^2\mu\rho)}{b_c^2-(b^2)^2} \pm \\
    \sqrt{ (g^{22} - 2g^{12}\lambda + g^{11}\lambda^2) \left( \frac{-4c^2(b^2-b^1\lambda)^2\mu\rho}{g^2} + (g^{22} - 2g^{12}\lambda + g^{11}\lambda^2)\left(\frac{(b_1)^2 (|\pmb{B}|^2+c^2\mu\rho)}{b_c^2-(b^2)^2}\right)^2 \right)  } \Biggr]   
\end{multline}

Where the speed of sound, $c$, is defined as:

\begin{equation}
    c = \frac{\sqrt{PP_e + \rho^2 P_\rho} }{\rho}
\end{equation}

It is possible to demonstrate graphically or numerically that there will not always be four real solutions to equation \eqref{lastFour}. In some situations, the solutions will be complex. Thus the type of the system will be hyperbolic or elliptic depending on the solution.

In section \ref{ProblemSetting} the setting of the problem is described qualitatively along with some expected features of the solution. Section \ref{prelim} describes the geometric machinery necessary to derive system \eqref{TheEq}. Section \ref{conical} describes the conical assumption imposed upon the unknowns. Section \ref{MHD_EQs} introduces the Ideal Magnetohydrodynamic equations which govern electrically conducting fluid flow, from which equations \eqref{TheEq} are derived. Following that, the projected equations are derived in section \ref{derivation}. Sections \ref{hyperbolic} and \ref{valsvects} discuss the type of system \eqref{TheEq} based on its eigenvalues.

\section{Problem Setting}\label{ProblemSetting}

The cone of arbitrary cross section is considered to be infinite and at an angle of attack relative to the free stream. Though conical electrically conducting flows have not been studied in detail, they are reasonably expected to be overall qualitatively similar to a non-conducting flow, with features such as an attached bow shock wave, crossflow streamlines which wrap around the body, and two or more body shocks which are caused by the crossflow briefly going supersonic as depicted in Figure \ref{crossShocks} \cite{sriThesis,ShockFreeCrossFlow,NASA_con,RemarksConFlow}. One should expect that velocity and temperature gradients, especially inside the shock wave and close to the body, will be flattened out compared to the non-conducting counterpart \cite{ExpResults,Prospects}. The shock wave angle should also increase \cite{NumSimMHD}. These effects result in a large part due to the Lorentz force which naturally opposes the fluid motion \cite{ExpResults,Prospects}. This force is stronger on faster moving fluid elements than on slower moving elements which flattens out velocity gradients, and the overall slower fluid requires stream tubes to increase in size in order to transport the same quantities. This effect has been shown to reduce conductive heating and heating due to skin friction and thus has potential for solving one of the primary problems in hypersonic design \cite{NumSimMHD,ExpResults}.


\begin{figure}
    \centering
    \includegraphics{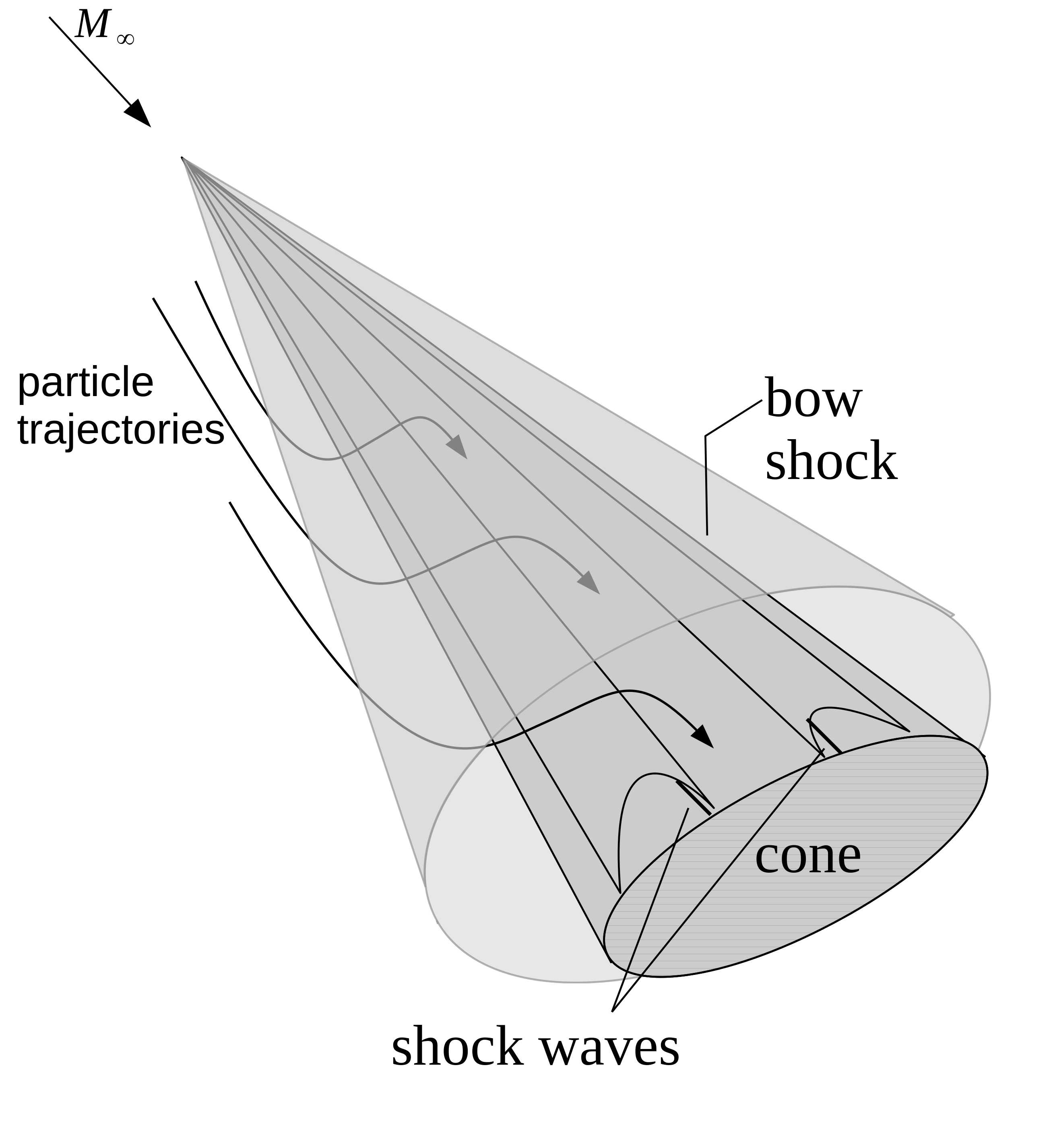}
    \caption{Supersonic infinite cones with elliptic cross section. Shock wave formation and particle trajectories are shown.}
    \label{crossShocks}
\end{figure}

Such a flow is said to be conical if there exists a point in the domain such that along any line that goes through this point, the flow properties (density, velocity, energy, magnetic field, etc) do not change \cite{RemarksConFlow,sriThesis}. Effectively, this means that if the origin is set to be the tip of the cone, then the solution has no ``$r$'' dependency, where $r$ is the distance from the origin. This type of flow can best be studied by taking a spherical slice out of the domain centered on the origin and projecting the vector fields onto that sphere as shown in Figure \ref{sphere}. A solution obtained on this spherical shell of a given radius will thus be valid on a shell of any other radius so that the flow in the whole of the 3D domain is accounted for.

\begin{figure}
    \centering
    \includegraphics{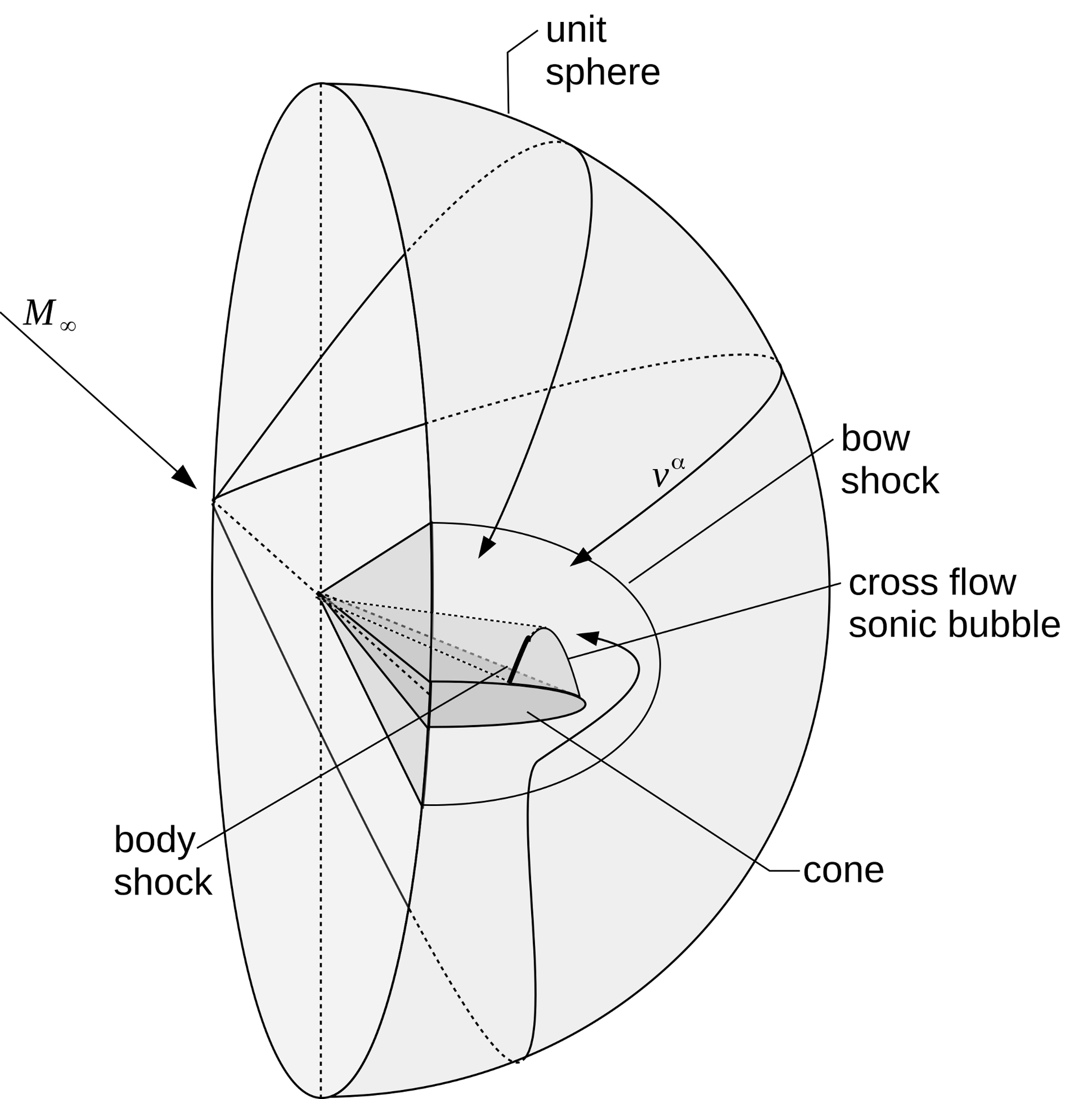}
    \caption{Problem setting sliced by a sphere with the velocity projected onto the surface giving the crossflow streamlines.}
    \label{sphere}
\end{figure}

Another interesting feature of this flow is that the governing system of partial differential equations can change type multiple times within the domain. As stated previously, the system is hyperbolic or elliptic depending on the solution. The changing back and forth of the type throughout the domain as well as regions of different types sharing boundaries must be accounted for in the theory and numerical solving of the governing equations.

\section{Geometric Preliminaries}\label{prelim}

Consider a 3D Euclidean space characterized by metric tensor $G_{ij}$ and coordinates $x^i$. Embedded in this 3D space is a 2D spherical subspace characterized by the metric tensor $\tilde{g}_{\alpha\beta}$ and coordinates $\xi^\alpha$ (in this article the convention is adopted that Latin indices such as $i,j$ take on values from 1 to 3 and Greek indices such as $\alpha,\beta$ take on values from 1 to 2). For such a subspace there is the relationship:

\begin{equation}
    \tilde{g}_{\alpha\beta} = G_{ij}F^i_\alpha F^j_\beta
\end{equation}

where the projection factors are given by:

\begin{equation}
    F^i_\alpha = \frac{\partial x^i}{\partial \xi^\alpha}
\end{equation}

and

\begin{equation}
    F^\alpha_i = \tilde{g}^{\alpha\beta}G_{ij}F^j_\beta
\end{equation}

A tensor in the embedding space can be projected onto the sphere using the projection factors, such as:

\begin{equation}
    \tilde{w}^\alpha = F_i^\alpha W^i, \phantom{m} \tilde{w}^{\alpha\beta} = F_i^\alpha F_j^\beta W^{ij}, \phantom{m} \tilde{w}^{\alpha\beta}_\nu = F_i^\alpha F_j^\beta F^k_\nu W^{ij}_k, \phantom{m} \text{etc.}
\end{equation}

It is convenient to treat the three dimensional embedding space as having the two subspace coordinates and a radial coordinate as its three coordinates. That is $\pmb{x}=(\xi^1,\xi^2,r)$. The $r$ coordinate is orthogonal to the other coordinates so the metric tensor of the embedding space in matrix form would be:

\begin{equation}
    G_{ij} = \left[ \begin{smallmatrix} \cdot&\cdot&0\\ \cdot&\cdot&0 \\ 0&0&1 \end{smallmatrix} \right], \phantom{m} 1\leq i,j \leq 3
\end{equation}

and that of the embedded subspace:

\begin{equation}
    \tilde{g}_{\alpha\beta}= G_{\alpha\beta}, \phantom{m} 1\leq\alpha,\beta\leq 2
\end{equation}

\begin{remark}
Note that though traditional spherical coordinates $\theta$ and $\phi$ on the surface of the sphere would be a valid choice of coordinates, one is not restricted to them. For this topic, one can consider any two surface coordinates and a radial one. This allows for the possibility of the coordinate lines being aligned with the surface of the cone (as shown in Figure \ref{sampCoords}) even if it has an irregular cross section. In the case of a numerical solution using a structured mesh, the coordinate lines can be defined to follow the mesh lines and simplify some calculations. 
\end{remark}

\begin{figure}
    \centering
    \includegraphics{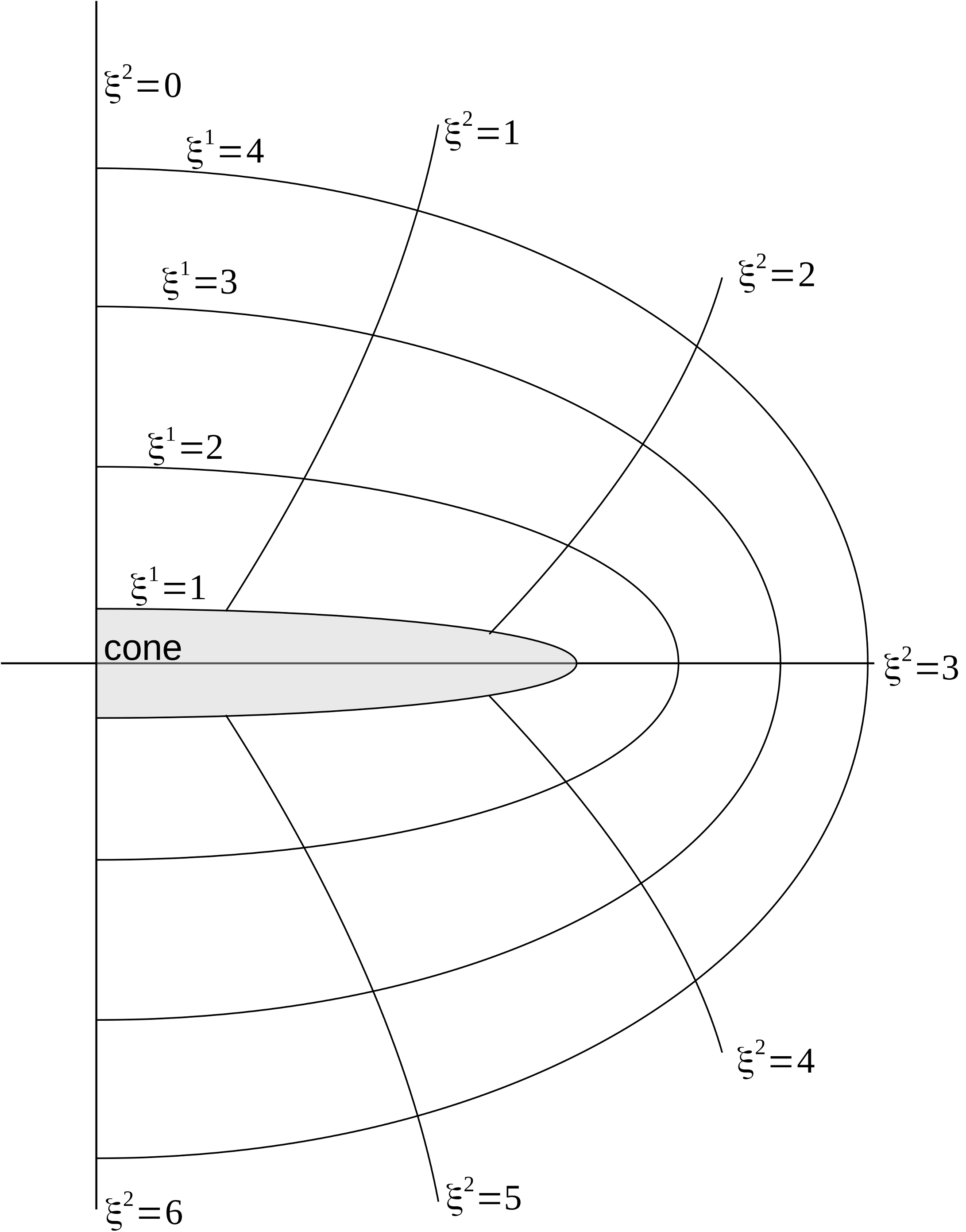}
    \caption{Example of coordinate lines which conform to the shape of the body and are not necessarily orthogonal. }
    \label{sampCoords}
\end{figure}

Any vector $\tilde{w}^\alpha$ defined at a point in the subspace will have a length defined by:

\begin{equation}
    |\tilde{\pmb{w}}|^2=\tilde{g}_{\alpha\beta}\tilde{w}^\alpha\tilde{w}^\beta
\end{equation}

Since this subspace is defined to be the surface of a sphere, distances will scale proportional to the radius of the subspace, $r$, giving:

\begin{equation}
    |\tilde{\pmb{w}}|^2=\tilde{g}_{\alpha\beta}\tilde{w}^\alpha\tilde{w}^\beta = r^2g_{\alpha\beta}\tilde{w}^\alpha\tilde{w}^\beta
\end{equation}

where the $r$ dependency has been separated out of the metric tensor. This implies that $\tilde{g}_{\alpha\beta}=r^2g_{\alpha\beta}$ (also $\tilde{g}^{\alpha\beta}=\frac{1}{r^2}g^{\alpha\beta}$) and that $g_{\alpha\beta}$ is a function of $\xi^1$ and $\xi^2$ only. This leads us to define a new representation of the vector where $w^\alpha=r\tilde{w}^\alpha$ and also $w_\alpha=\frac{1}{r}\tilde{w}_\alpha$. Using this definition:

\begin{equation}\label{rescaledLength}
    |\tilde{\pmb{w}}|^2=\tilde{g}_{\alpha\beta}\tilde{w}^\alpha\tilde{w}^\beta = r^2g_{\alpha\beta}\tilde{w}^\alpha\tilde{w}^\beta = g_{\alpha\beta}w^\alpha w^\beta
\end{equation}

In particular, equation \eqref{rescaledLength} says the magnitude of the surface components of a vector does not change as you scale in r. This representation, with the $r$ dependency shifted from the metric onto the vector components, can be used for any vector.

\section{Conical assumption}\label{conical}

This article describes the Ideal MHD equations subject to the conical assumption on all the dependent variables (density, velocity, energy, and magnetic field). 

\begin{definition}
A quantity is said to be conical if the covariant derivative in the $r$ direction is identically zero.
\end{definition}

For scalar quantities such as $\rho$ and $E$, this means that the partial derivative with respect to $r$ is zero. For higher order tensorial quantities it is not so simple. Because the basis for the vectors is not uniform, it is possible for the components of a vector to change, but for the vector to remain the same, and conversely for the vector to change, but the components to remain the same. Therefore the covariant derivative must be used, which accounts for the changing of the underlying coordinate basis. 

Consider a vector, $\pmb{W}$, in the 3D embedding space. It has 3 components; two corresponding to the spherical subspace and one radial component, that is $\pmb{W}=\left[ \begin{smallmatrix} \tilde{w}^1 & \tilde{w}^2 & W^3 \end{smallmatrix} \right]$. If $\pmb{W}$ is conical, then all components of the covariant derivative in the $x^3$ (or $r$) direction are identically zero. Mathematically, that is:

\begin{equation}
W^i_{|3}=0, \phantom{m}\forall i
\end{equation}

Inserting the full expression for the covariant derivative gives:

\begin{equation*}
    W^i_{|3} = \frac{\partial W^i}{\partial x^3} + \Gamma\indices{_j^i_3}W^j
\end{equation*}

\begin{equation*}
   = \frac{\partial W^i}{\partial x^3} + \frac{G^{ik}}{2}\left[ 
\frac{\partial G_{kj}}{\partial x^3} + \frac{\partial G_{k3}}{\partial x^j} - \frac{\partial G_{j3}}{\partial x^k} \right]W^j
\end{equation*}

Because of the form of the metric, the last two terms in the Christoffel symbol are identically zero, giving:

\begin{equation}
   = \frac{\partial W^i}{\partial x^3} + \frac{G^{ik}}{2}\left[ 
\frac{\partial G_{kj}}{\partial x^3} \right]W^j
\end{equation}

Examining this expression, when $i=3$ it becomes:

\begin{equation}
    \frac{\partial W^3}{\partial x^3} = \frac{\partial W^3}{\partial r} =0
\end{equation}

And otherwise:

\begin{equation}
   = \frac{\partial \tilde{w}^\alpha}{\partial x^3} + \frac{\tilde{g}^{\alpha\nu}}{2}\left[ 
\frac{\partial \tilde{g}_{\nu\beta}}{\partial x^3} \right]\tilde{w}^\beta
\end{equation}

Plugging in the components with the shifted $r$ dependency gives:

\begin{equation*}
   = \frac{\partial }{\partial r}\left(\frac{1}{r}w^\alpha\right) + \frac{g^{\alpha\nu}}{2r^2}\left[ 
\frac{\partial }{\partial r}\left(r^2g_{\nu\beta}\right) \right]\frac{1}{r}w^\beta
\end{equation*}

\begin{equation*}
   = \frac{1}{r}\frac{\partial }{\partial r}\left(w^\alpha\right) - \frac{1}{r^2}w^\alpha + \frac{g^{\alpha\nu}}{2r^2}\left[ 
2rg_{\nu\beta} \right]\frac{1}{r}w^\beta
\end{equation*}

\begin{equation*}
   = \frac{1}{r}\frac{\partial w^\alpha}{\partial r} - \frac{1}{r^2}w^\alpha + \frac{1}{r^2}\delta^{\alpha}_\beta w^\beta
\end{equation*}

\begin{equation*}
   = \frac{1}{r}\frac{\partial w^\alpha}{\partial r} - \frac{1}{r^2}w^\alpha + \frac{1}{r^2}w^\alpha
\end{equation*}

\begin{equation}
   = \frac{1}{r}\frac{\partial w^\alpha}{\partial r}=0
\end{equation}

Thus the conical assumption implies: 

\begin{equation}
   \frac{\partial W^3}{\partial r} = \frac{\partial w^\alpha}{\partial r} = 0
\end{equation}

This expression tells us that it is the rescaled components of the vector and not the original components which are independent of $r$. This is an important concept to keep in mind as the conical equations are derived.

\section{Ideal MHD equations}\label{MHD_EQs}

The Ideal MHD equations are given in a Cartesian setting for reference:

\begin{subequations}\label{CartesianNoPowell}
\begin{gather}
      \rho_t  + \nabla\cdot(\rho \pmb{V})=0
\\
      (\rho \pmb{V})_t  + \nabla\cdot\left[\rho \pmb{V}\otimes\pmb{V}+\left(P+\frac{|\pmb{B}|^2}{2\mu}\right)I-\frac{1}{\mu}\pmb{B}\otimes\pmb{B}\right]= \pmb{0}
\\
      (\rho E + \frac{|\pmb{B}|^2}{2\mu})_t  + \nabla\cdot\left[\left(\rho E+P+\frac{|\pmb{B}|^2}{\mu}\right)\pmb{V} - \frac{1}{\mu}(\pmb{V}\cdot\pmb{B})\pmb{B}\right]= 0
\\
      \pmb{B}_t + \nabla\cdot\left( \pmb{V}\otimes\pmb{B} - \pmb{B}\otimes\pmb{u} \right) = \pmb{0}
\end{gather}
\end{subequations}

The dependent variables are $\rho$, $\pmb{V}$, $e$, and $\pmb{B}$, and $P=P(\rho,e)$ is provided by a gas law to close the system. These equations represent the limit of infinite conductivity and zero viscosity in the fluid, assumptions referring to the case where flow induced electromagnetic effects overwhelm imposed fields and inertial effects overwhelm viscous effects. 

It is known that when these equations are put into quasilinear form the matrices are degenerate, causing a characteristic speed to be equal to zero. To correct this, Powell \cite{Powell} proposed an equivalent system of equations:

\begin{subequations}\label{Cartesian}
\begin{gather}
      \rho_t  + \nabla\cdot(\rho \pmb{V})=0
\\
      (\rho \pmb{V})_t  + \nabla\cdot\left[\rho \pmb{V}\otimes\pmb{V}+\left(P+\frac{|\pmb{B}|^2}{2\mu}\right)I-\frac{1}{\mu}\pmb{B}\otimes\pmb{B}\right]= -\frac{1}{\mu}\pmb{B}\nabla\cdot\pmb{B}
\\
      (\rho E + \frac{|\pmb{B}|^2}{2\mu})_t  + \nabla\cdot\left[\left(\rho E+P+\frac{|\pmb{B}|^2}{\mu}\right)\pmb{V} - \frac{1}{\mu}(\pmb{V}\cdot\pmb{B})\pmb{B}\right]= - \frac{1}{\mu}(\pmb{V}\cdot\pmb{B})\nabla\cdot\pmb{B} 
\\
      \pmb{B}_t + \nabla\cdot\left( \pmb{V}\otimes\pmb{B} - \pmb{B}\otimes\pmb{u} \right) = -\pmb{V}\nabla\cdot\pmb{B}
\end{gather}
\end{subequations}

The terms on the RHS of equation \ref{Cartesian} are Powell's source terms \cite{Sterck,JamesonMHD}. In a true solution, each of them will be equal to zero, and thus the system is unchanged. However, when this new system is put in quasilinear form, the resulting matrices are full rank. The zero characteristic speed is replaced by the velocity of the fluid and the corresponding eigenvector does not interfere with the other seven \cite{Powell}.

To use the machinery of tensor calculus to project the equations onto a unit sphere, it is convenient to have the coordinate free form for the contravariant components:

\begin{subequations}\label{transient}
\begin{gather}
    \left(\sqrt{G}\rho\right)_t  + \left(\rho\sqrt{G}V^j\right)_{|j}=0 
    \\
   \left(\rho\sqrt{G}V^i\right)_t  + \left(\sqrt{G}\left[\rho V^iV^j - \frac{1}{\mu}B^iB^j + G^{ij}(P + \frac{|\pmb{B}|^2}{2\mu})\right]\right)_{|j} = -\frac{\sqrt{G}}{\mu}B^iB^j_{|j} 
    \\
    \left(\sqrt{G}\left[\rho E + \frac{|\pmb{B}|^2}{2\mu} \right]\right)_t \nonumber \\ + 
        \left( \sqrt{G}\left[\left(\rho E+P+\frac{|\pmb{B}|^2}{\mu}\right)V^j
        - \frac{1}{\mu}(\pmb{V}\cdot\pmb{B})B^j\right] \right)_{|j} \nonumber \\ = -\frac{\sqrt{G}}{\mu}(\pmb{V}\cdot\pmb{B})B^j_{|j} 
    \\
    B^i_t + (V^jB^i-V^iB^j)_{|j} = -V^iB^j_{|j} 
\end{gather}
\end{subequations}

The notation $(\cdot)_{|j}$ refers to the covariant derivative. The steady problem is the object of consideration (and furthermore time dependency is incompatible with the conical assumption), so the time derivative terms will all be set to zero, leaving Equation \eqref{steady}:

\begin{subequations}\label{steady}
\begin{gather}
    \left(\rho\sqrt{G}V^j\right)_{|j}=0 \label{3mass}
    \\
   \left(\sqrt{G}\left[\rho V^iV^j - \frac{1}{\mu}B^iB^j + G^{ij}(P + \frac{|\pmb{B}|^2}{2\mu})\right]\right)_{|j} = -\frac{\sqrt{G}}{\mu}B^iB^j_{|j} \label{3mom}
    \\
    \left( \sqrt{G}\left[\left(\rho E+P+\frac{|\pmb{B}|^2}{\mu}\right)V^j - \frac{1}{\mu}(\pmb{V}\cdot\pmb{B})B^j\right] \right)_{|j} = -\frac{\sqrt{G}}{\mu}(\pmb{V}\cdot\pmb{B})B^j_{|j} \label{3energy}
    \\
    (V^jB^i-V^iB^j)_{|j} = -V^iB^j_{|j} \label{3mag}
\end{gather}
\end{subequations}

The task is to derive an equivalent set of equations on the surface of a sphere for a conical solution.

\section{Derivation of conical equations}\label{derivation}

For the projection of the equations, the following relations are necessary which involve the various elements of the different spaces:

\begin{equation}\label{metric1}
g_{\alpha\beta} = \frac{1}{r^2}\tilde{g}_{\alpha\beta} = \frac{1}{r^2}G_{\alpha\beta},\phantom{m}
	G_{ij} = \left[ \begin{smallmatrix} \cdot&\cdot&0\\ \cdot&\cdot&0 \\ 0&0&1 \end{smallmatrix} \right]
\end{equation}

\begin{equation}\label{metric2}
g^{\alpha\beta} = r^2\tilde{g}^{\alpha\beta} 
\end{equation}

\begin{equation}
g = \frac{1}{r^4}\tilde{g} = \frac{1}{r^4}G
\end{equation}

\begin{equation}
\sqrt{g} = \frac{1}{r^2}\sqrt{\tilde{g}} = \frac{1}{r^2}\sqrt{G}
\end{equation}

\begin{equation}
\tilde{g}_{\alpha\beta} = G_{ij}F^i_\alpha F^j_\beta
\end{equation}

\begin{equation}\label{switchMetrics}
\tilde{g}^{\alpha\beta}F^j_\beta = G^{ij}F_i^\alpha 
\end{equation}

\begin{equation}\label{projRank1}
\tilde{v}^\alpha = F^\alpha_iV^i,\phantom{m} F^i_\alpha = \frac{\partial x^i}{\partial\xi^\alpha}
\end{equation}

\begin{equation}
    F^\beta_jF^i_\beta = \delta^i_j - N^i_j,\phantom{m} N^i_j \equiv \delta^i_3\delta^3_j
\end{equation}

\begin{equation}
\tilde{v}_\alpha = rv_\alpha
\end{equation}

\begin{equation}
\tilde{v}^\alpha = \frac{1}{r}v^\alpha
\end{equation}

\begin{equation}
    \overset{(g)}{\Gamma}\indices{_\gamma^\alpha_\nu} = \overset{(\tilde{g})}{\Gamma}\indices{_\gamma^\alpha_\nu} = F^\alpha_lF^i_\gamma F^k_\nu \overset{(G)}{\Gamma}\indices{_i^l_k}
\end{equation}

\begin{equation}\label{contractChris}
    \overset{(G)}{\Gamma}\indices{_i^j_j} = \frac{1}{\sqrt{G}}\frac{\partial \sqrt{G}}{\partial x^i}
\end{equation}

\subsection{Source terms}

The Powell source terms are all proportional to $B^j_{|j}$, and so it is necessary to have an expression for this on the surface of the sphere. The projection begins by considering $b^\beta_{||\beta}$, which is given by:

\begin{equation}
    b^\beta_{||\beta} = \frac{\partial b^\beta}{\partial\xi^\beta} + \overset{(g)}{\Gamma}\indices{_\gamma^\nu_\nu}b^\gamma = \frac{\partial }{\partial\xi^\beta}( r\tilde{b}^\beta ) + r\overset{(g)}{\Gamma}\indices{_\gamma^\nu_\nu}\tilde{b}^\gamma
\end{equation}

This can be written in terms of the 3D elements:

\begin{equation*}
    = \frac{\partial }{\partial\xi^\beta}( rF^\beta_jB^j ) + rF^i_\gamma F^\nu_l F^m_\nu\overset{(G)}{\Gamma}\indices{_i^l_m}F^\gamma_kB^k
\end{equation*}

Moving the projection factors and changing derivative:

\begin{equation*}
    = rF^\beta_jF^n_\beta\frac{\partial }{\partial x^n}( B^j ) + rF^i_\gamma F^\nu_l F^m_\nu F^\gamma_k\overset{(G)}{\Gamma}\indices{_i^l_m}B^k
\end{equation*}

\begin{equation*}
    = r(\delta^n_j-N^n_j)\frac{\partial }{\partial x^n}( B^j ) + r(\delta^m_l-N^m_l)(\delta^i_k-N^i_k)\overset{(G)}{\Gamma}\indices{_i^l_m}B^k
\end{equation*}

\begin{equation*}
    = r(\frac{\partial B^j}{\partial x^j} - \frac{\partial B^3}{\partial x^3}) + r(\delta^m_l\delta^i_k - \delta^m_l N^i_k - N^m_l\delta^i_k + N^m_lN^i_k)\overset{(G)}{\Gamma}\indices{_i^l_m}B^k
\end{equation*}

\begin{equation*}
    = r\frac{\partial B^j}{\partial x^j} + r(\overset{(G)}{\Gamma}\indices{_k^l_l}B^k - \overset{(G)}{\Gamma}\indices{_3^l_l}B^3 - \overset{(G)}{\Gamma}\indices{_k^3_3}B^k + \overset{(G)}{\Gamma}\indices{_3^3_3}B^3)
\end{equation*}

The first term plus the first term in parenthesis is $r$ times the divergence of $\pmb{B}$. The third and fourth terms in parenthesis are both zero. This leaves:

\begin{equation*}
    = rB^j_{|j} - r\overset{(G)}{\Gamma}\indices{_3^l_l}B^3
\end{equation*}

using equation \eqref{contractChris}:

\begin{equation*}
    = rB^j_{|j} - r\frac{1}{\sqrt{G}}\frac{\partial \sqrt{G}}{\partial x^3} B^3
\end{equation*}

rescaling:

\begin{equation*}
    = rB^j_{|j} - r\frac{1}{r^2\sqrt{g}}\frac{\partial r^2\sqrt{g}}{\partial r} B^3
\end{equation*}

\begin{equation*}
    = rB^j_{|j} - r\frac{2r}{r^2\sqrt{g}}\sqrt{g} B^3
\end{equation*}

\begin{equation*}
    = rB^j_{|j} - 2 B^3
\end{equation*}

Finally, we have:

\begin{equation}
    B^j_{|j} = \frac{ b^\beta_{||\beta} + 2 B^3 }{r}
\end{equation}

The remaining ``$r$'' in this expression will be cancelled out for all the equations in \eqref{steady}.

\subsection{Mass equation}

The focus now shifts to the mass equation \eqref{3mass}. The LHS of this equation is the contracted covariant derivative (divergence) of a rank 1 relative tensor of weight 1 as defined by \cite{Lovelock}. This means that it carries with it the square root of the determinant of the metric tensor raised to the first power, and that when it transforms from one coordinate system to the other, it changes metric determinants as well. The contracted covariant derivative of such a tensor is given by the following:

Let $W\indices{^{j}}$ be a rank 1 contravariant relative tensor of weight 1, such as $\rho\sqrt{G}V^j$. Then the contracted covariant derivative is given by:

\begin{equation}
W\indices{^{j}_{|j}} = \frac{\partial W^{j}}{\partial x^j} 
\end{equation}

And likewise for a surface tensor of the same type and weight, and its rescaled components:

\begin{equation}
\tilde{w}\indices{^{\beta}_{||\beta}} = \frac{\partial \tilde{w}^{\beta}}{\partial \xi^\beta} 
\end{equation}

\begin{equation}
w\indices{^{\beta}_{||\beta}} = \frac{\partial w^{\beta}}{\partial \xi^\beta} 
\end{equation}

In this case $\tilde{w}^\beta=\rho\sqrt{\tilde{g}}\tilde{v}^\beta$ and $w^\beta=\rho\sqrt{g}v^\beta$. 

Using previous relations, the surface divergence for the rescaled components can be rewritten as:

\begin{equation*}
w\indices{^{\beta}_{||\beta}} = \frac{\partial }{\partial \xi^\beta}\left(\frac{1}{r}\tilde{w}^\beta\right) = \frac{1}{r}\frac{\partial }{\partial \xi^\beta}\left(F^\beta_iW^i\right) 
\end{equation*}

\begin{equation*}
= \frac{1}{r}F^\beta_i\frac{\partial W^i}{\partial \xi^\beta} + \frac{1}{r}W^i\frac{\partial }{\partial \xi^\beta}F^\beta_i
\end{equation*}

\begin{equation*}
= \frac{1}{r}F^\beta_iF^j_\beta\frac{\partial W^i}{\partial x^j} + 0
\end{equation*}

\begin{equation*}
= \frac{1}{r}\left(\delta^i_j - N^i_j\right)\frac{\partial W^i}{\partial x^j}
\end{equation*}

\begin{equation*}
= \frac{1}{r}\frac{\partial W^j}{\partial x^j} - \frac{1}{r}\frac{\partial W^3}{\partial x^3}
\end{equation*}

Thus:

\begin{equation}
w\indices{^{\beta}_{||\beta}}= \frac{1}{r}W\indices{^j_{|j}} - \frac{1}{r}\frac{\partial W^3}{\partial x^3}
\end{equation}

or

\begin{equation}\label{projRank1Rel}
    w^\beta_{||\beta} + \frac{1}{r}\frac{\partial W^3}{\partial r} = \frac{1}{r}W^j_{|j}
\end{equation}

For the mass equation, $W^i_{|j}$ is simply equal to zero because there is no source term. Plugging in the expressions $w^\beta=\rho\sqrt{g}v^\beta$ and $W^3=\rho\sqrt{G}V^3$ gives:

\begin{equation*}
    (\rho\sqrt{g}v^\beta)_{||\beta} + \frac{1}{r}\frac{\partial}{\partial r}(\rho\sqrt{G}V^3) = 0
\end{equation*}

Thus:

\begin{equation*}
    (\rho\sqrt{g}v^\beta)_{||\beta} + \frac{1}{r}\frac{\partial}{\partial r}(\rho r^2\sqrt{g}V^3) = 0
\end{equation*}

\begin{equation*}
    (\rho\sqrt{g}v^\beta)_{||\beta} + \frac{2r}{r}\rho \sqrt{g}V^3 = 0
\end{equation*}

And finally:

\begin{equation}
    (\rho\sqrt{g}v^\beta)_{||\beta} + 2\rho \sqrt{g}V^3 = 0
\end{equation}

which is equation \eqref{mass}.

\subsection{Energy equation}

The LHS of the energy equation (equation \eqref{3energy}) is also the contracted covariant derivative of a rank 1 relative tensor of weight 1. Thus the result is again equation \eqref{projRank1Rel}:

\begin{equation*}
    w^\beta_{||\beta} + \frac{1}{r}\frac{\partial W^3}{\partial r} = \frac{1}{r}W^j_{|j}
\end{equation*}

with

\begin{equation}
    W^j =  \sqrt{G}\left[\left(\rho E+P+\frac{|\pmb{B}|^2}{\mu}\right)V^j - \frac{1}{\mu}(\pmb{V}\cdot\pmb{B})B^j\right]
\end{equation}

and 

\begin{equation}
    w^\beta = \sqrt{g}\left[\left(\rho E+P+\frac{|\pmb{B}|^2}{\mu}\right)v^\beta - \frac{1}{\mu}(\pmb{V}\cdot\pmb{B})b^\beta\right]
\end{equation}

For this equation, there is a source term, so:

\begin{equation*}
    W^j_{|j} = -\frac{\sqrt{G}}{\mu}(\pmb{V}\cdot\pmb{B})B^k_{|k}
\end{equation*}

\begin{equation*}
    = -\frac{\sqrt{G}}{\mu}(\pmb{V}\cdot\pmb{B})\frac{ b^\beta_{||\beta} + 2B^3 }{r}
\end{equation*}

\begin{equation*}
    = -r^2\frac{\sqrt{g}}{\mu}(\pmb{V}\cdot\pmb{B})\frac{ b^\beta_{||\beta} + 2 B^3 }{r} 
\end{equation*}

\begin{equation}
    = -r\frac{\sqrt{g}}{\mu}(\pmb{V}\cdot\pmb{B}) (b^\beta_{||\beta} + 2 B^3 )
\end{equation}

Thus:

\begin{equation}
    w^\beta_{||\beta} + \frac{1}{r}\frac{\partial W^3}{\partial r} = -\frac{1}{r}r\frac{\sqrt{g}}{\mu}(\pmb{V}\cdot\pmb{B}) (b^\beta_{||\beta} + 2 B^3 )
\end{equation}

which leads to:

\begin{equation}
    w^\beta_{||\beta} + \frac{1}{r}\frac{\partial W^3}{\partial r} = -\frac{\sqrt{g}}{\mu}(\pmb{V}\cdot\pmb{B}) (b^\beta_{||\beta} + 2 B^3 )
\end{equation}

Plugging in the expressions for $w^\beta$ and $W^3$ and following the same procedure to reduce the second term on the LHS as for the mass equation results in equation \eqref{energy}.

\subsection{Momentum equation}

The projection continues by considering the momentum equation (equation \eqref{3mom}). The LHS is the contracted covariant derivative (divergence) of a rank 2 relative tensor of weight 1 as defined by \cite{Lovelock}. The contracted covariant derivative of which is given by the following:

Let $W\indices{^{ij}}$ be a rank 2 contravariant relative tensor of weight 1, such as:

\begin{equation*}
    \sqrt{G}\left[\rho V^iV^j - \frac{1}{\mu}B^iB^j + G^{ij}(P + \frac{|\pmb{B}|^2}{2\mu})\right]
\end{equation*}

Then the contracted covariant derivative is given by:

\begin{equation}
W\indices{^{ij}_{|j}} = \frac{\partial W^{ij}}{\partial x^j} + \overset{(G)}{\Gamma}\indices{_h^i_k}W^{hk}
\end{equation}

And likewise for a surface tensor of the same type and weight and its rescaled components.

\begin{equation}
\tilde{w}\indices{^{\alpha\beta}_{||\beta}} = \frac{\partial \tilde{w}^{\alpha\beta}}{\partial \xi^\beta} + \overset{(\tilde{g})}{\Gamma}\indices{_\gamma^\alpha_\nu}\tilde{w}^{\gamma\nu}
\end{equation}

\begin{equation}\label{surfDivRank2}
w\indices{^{\alpha\beta}_{||\beta}} = \frac{\partial w^{\alpha\beta}}{\partial \xi^\beta} + \overset{(g)}{\Gamma}\indices{_\gamma^\alpha_\nu}w^{\gamma\nu}
\end{equation}

Where the Christoffel symbols, $\tilde{w}$, and $w$ are defined in terms of the respective metric tensors. 

In analogy to the projection of the continuity and energy equations, previous relations are plugged into the surface divergence expression for the rescaled tensor. We thus proceed:

To begin, an expression for the Christoffel symbol defined by the rescaled metric which has no $r$ dependency is found. It is given by:

\begin{equation*}
\overset{(g)}{\Gamma}\indices{_\gamma^\alpha_\nu} = \frac{ g^{\alpha\beta}}{2}\left[ 
\frac{\partial g_{\beta\gamma}}{\partial\xi^\nu} + \frac{\partial g_{\beta\nu}}{\partial\xi^\gamma} - \frac{\partial g_{\gamma\nu}}{\partial\xi^\beta} \right] = \frac{\tilde{g}^{\alpha\beta}}{2}\left[ 
\frac{\partial\tilde{g}_{\beta\gamma}}{\partial\xi^\nu} + \frac{\partial\tilde{g}_{\beta\nu}}{\partial\xi^\gamma} - \frac{\partial\tilde{g}_{\gamma\nu}}{\partial\xi^\beta} \right] = \overset{(\tilde{g})}{\Gamma}\indices{_\gamma^\alpha_\nu}
\end{equation*}

because the $r^2$'s from Equations \eqref{metric1} and \eqref{metric2} cancel. Proceeding with the projection:

\begin{equation*}
 = \frac{\tilde{g}^{\alpha\beta}}{2}\left[ 
\frac{\partial }{\partial\xi^\nu}(G_{hk}F^h_\beta F^k_\gamma) + \frac{\partial }{\partial\xi^\gamma}(G_{hk}F^h_\beta F^k_\nu) - \frac{\partial }{\partial\xi^\beta}(G_{hk}F^h_\gamma F^k_\nu) \right]
\end{equation*}

Now pulling out the projection factors and changing derivatives:

\begin{equation*}
 = \frac{\tilde{g}^{\alpha\beta}}{2}\left[ 
F^h_\beta F^k_\gamma F^i_\nu\frac{\partial }{\partial x^i}(G_{hk}) + F^h_\beta F^k_\nu F^i_\gamma\frac{\partial }{\partial x^i}(G_{hk}) - F^h_\gamma F^k_\nu F^i_\beta\frac{\partial }{\partial x^i}(G_{hk}) \right]
\end{equation*}

Adjusting indices on the last term:

\begin{equation*}
 = \frac{\tilde{g}^{\alpha\beta}}{2}\left[ 
F^h_\beta F^k_\gamma F^i_\nu\frac{\partial }{\partial x^i}(G_{hk}) + F^h_\beta F^k_\nu F^i_\gamma\frac{\partial }{\partial x^i}(G_{hk}) - F^i_\gamma F^k_\nu F^h_\beta\frac{\partial }{\partial x^h}(G_{ik}) \right]
\end{equation*}

Then pulling out the common projection factor:

\begin{equation*}
 = \frac{\tilde{g}^{\alpha\beta}}{2}F^h_\beta\left[ 
F^k_\gamma F^i_\nu\frac{\partial }{\partial x^i}(G_{hk}) + F^k_\nu F^i_\gamma\frac{\partial }{\partial x^i}(G_{hk}) - F^i_\gamma F^k_\nu\frac{\partial }{\partial x^h}(G_{ik}) \right]
\end{equation*}

\begin{equation*}
 = \frac{G^{hl}}{2}F^\alpha_l\left[ 
F^k_\gamma F^i_\nu\frac{\partial }{\partial x^i}(G_{hk}) + F^k_\nu F^i_\gamma\frac{\partial }{\partial x^i}(G_{hk}) - F^i_\gamma F^k_\nu\frac{\partial }{\partial x^h}(G_{ik}) \right]
\end{equation*}

Adjusting indices on the first term so all projection factors can be pulled out:

\begin{equation*}
 = \frac{G^{hl}}{2}F^\alpha_l\left[ 
F^i_\gamma F^k_\nu\frac{\partial }{\partial x^k}(G_{hi}) + F^k_\nu F^i_\gamma\frac{\partial }{\partial x^i}(G_{hk}) - F^i_\gamma F^k_\nu\frac{\partial }{\partial x^h}(G_{ik}) \right]
\end{equation*}

\begin{equation*}
 = F^\alpha_lF^i_\gamma F^k_\nu\frac{G^{hl}}{2}\left[ 
\frac{\partial }{\partial x^k}(G_{hi}) + \frac{\partial }{\partial x^i}(G_{hk}) - \frac{\partial }{\partial x^h}(G_{ki}) \right]
\end{equation*}

\begin{equation*}
 = F^\alpha_lF^i_\gamma F^k_\nu\overset{(G)}{\Gamma}\indices{_i^l_k}
\end{equation*}

Thus:

\begin{equation}
 \overset{(g)}{\Gamma}\indices{_\gamma^\alpha_\nu} = \overset{(\tilde{g})}{\Gamma}\indices{_\gamma^\alpha_\nu} = F^\alpha_lF^i_\gamma F^k_\nu \overset{(G)}{\Gamma}\indices{_i^l_k}
\end{equation}

Which shows that the Christoffel symbol projects just like a tensor in this case. The full momentum equation can now be addressed.

\begin{equation*}
\frac{\partial w^{\alpha\beta}}{\partial \xi^\beta} + \overset{(g)}{\Gamma}\indices{_\gamma^\alpha_\nu}w^{\gamma\nu} = \frac{\partial \tilde{w}^{\alpha\beta}}{\partial \xi^\beta} + \overset{(\tilde{g})}{\Gamma}\indices{_\gamma^\alpha_\nu}\tilde{w}^{\gamma\nu}
\end{equation*}

\begin{equation*}
= \frac{\partial}{\partial\xi^\beta}(F^\alpha_iF^\beta_jW^{ij}) + F^\gamma_iF^\nu_jF^m_\gamma F^\alpha_lF^k_\nu \overset{(G)}{\Gamma}\indices{_m^l_k} W^{ij}
\end{equation*}

The projection factors in the first term are pulled out and the chain rule is used to change the derivative. Indices l and i in the second term are also switched.

\begin{equation*}
 = F^\alpha_iF^\beta_jF^f_\beta\frac{\partial}{\partial x^f}(W^{ij}) + F^\alpha_iF^\gamma_lF^m_\gamma F^\nu_jF^k_\nu \overset{(G)}{\Gamma}\indices{_m^i_k} W^{lj}
\end{equation*}

\begin{equation*}
 = F^\alpha_i\left[ (\delta^f_j-N^f_j)\frac{\partial}{\partial x^f}(W^{ij}) + (\delta^m_l-N^m_l)(\delta^k_j-N^k_j) \overset{(G)}{\Gamma}\indices{_m^i_k} W^{lj} \right]
\end{equation*}

\begin{equation*}
 = F^\alpha_i\left[ \frac{\partial}{\partial x^j}(W^{ij}) - \frac{\partial}{\partial x^3}(W^{i3}) + (\delta^m_l\delta^k_j-\delta^m_lN^k_j  -N^m_l\delta^k_j+N^m_lN^k_j) \overset{(G)}{\Gamma}\indices{_m^i_k} W^{lj} \right]
\end{equation*}

\begin{equation*}
 = F^\alpha_i\left[ \left(\frac{\partial}{\partial x^j}(W^{ij}) + \overset{(G)}{\Gamma}\indices{_l^i_j}W^{lj}\right) - \frac{\partial}{\partial x^3}(W^{i3}) - \left( \overset{(G)}{\Gamma}\indices{_l^i_3}W^{l3} + \overset{(G)}{\Gamma}\indices{_3^i_j}W^{3j} \right)
+ \overset{(G)}{\Gamma}\indices{_3^i_3}W^{33} \right]
\end{equation*}

The first term in parenthesis is $W\indices{^{ij}_{|j}}$. The terms in the second set of parenthesis are equal due to the symmetry of $W^{ij}$ and the Christoffel symbol. The last term is equal to zero because of the form of the metric. With all of this the result is:

\begin{equation}\label{MomentumProjA}
 = F^\alpha_i\left[ W\indices{^{ij}_{|j}} - \frac{\partial}{\partial x^3}(W^{i3}) - 2\overset{(G)}{\Gamma}\indices{_3^i_j}W^{3j} \right]
\end{equation}

For clarity, each term is treated individually. The first term is:

\begin{equation*}
    F^\alpha_i W\indices{^{ij}_{|j}} = -F^\alpha_i \frac{\sqrt{G}}{\mu}B^iB^j_{|j}
\end{equation*}

\begin{equation*}
    = -\frac{\sqrt{G}}{\mu}\frac{1}{r}b^\alpha\frac{ b^\beta_{||\beta} + 2 B^3 }{r}
\end{equation*}

\begin{equation*}
    = -\frac{r^2\sqrt{g}}{\mu}\frac{1}{r}b^\alpha\frac{ b^\beta_{||\beta} + 2 B^3 }{r}
\end{equation*}

\begin{equation}
    = -\frac{\sqrt{g}}{\mu}b^\alpha(b^\beta_{||\beta} + 2 B^3)
\end{equation}

The second term in equation \eqref{MomentumProjA} is:

\begin{equation*}
    -F^\alpha_i \frac{\partial}{\partial x^3}(W^{i3}) = -F^\alpha_i \frac{\partial}{\partial x^3}\left(\sqrt{G}\left[\rho V^iV^3 - \frac{1}{\mu}B^iB^3 + G^{i3}(P + \frac{|\pmb{B}|^2}{2\mu})\right]\right)
\end{equation*}

\begin{equation*}
    = -\frac{\partial}{\partial r}\left(r^2\sqrt{g}\left[\rho \frac{1}{r}v^\alpha V^3 - \frac{1}{\mu}\frac{1}{r}b^\alpha B^3\right]\right)
\end{equation*}

\begin{equation}
    = -\sqrt{g}\left[\rho v^\alpha V^3 - \frac{1}{\mu}b^\alpha B^3\right]
\end{equation}

The third term in equation \eqref{MomentumProjA} is:

\begin{equation*}
    -F^\alpha_i 2\overset{(G)}{\Gamma}\indices{_3^i_j}W^{3j} = -2F^\alpha_i\frac{G^{ih}}{2}\left[ 
\frac{\partial }{\partial x^3}(G_{hj}) + \frac{\partial }{\partial x^j}(G_{h3}) - \frac{\partial }{\partial x^h}(G_{j3}) \right]W^{3j}
\end{equation*}

using relation \eqref{switchMetrics} and \eqref{metric2} from above:

\begin{equation*}
    = -2F^h_\beta\frac{g^{\alpha\beta}}{2r^2}\left[ 
\frac{\partial }{\partial x^3}(G_{hj}) + \frac{\partial }{\partial x^j}(G_{h3}) - \frac{\partial }{\partial x^h}(G_{j3}) \right]W^{3j}
\end{equation*}

and since $F^h_\alpha$ acts like $\delta^h_\alpha$:

\begin{equation*}
    = -2\frac{g^{\alpha\beta}}{2r^2}\left[ 
\frac{\partial }{\partial x^3}(G_{\beta j}) + \frac{\partial }{\partial x^j}(G_{\beta 3}) - \frac{\partial }{\partial \xi^\beta}(G_{j3}) \right]W^{3j}
\end{equation*}

The second and third terms in brackets are both zero due to the form of the metric. The first term is zero when $j=3$. This leaves:

\begin{equation*}
    = -\frac{g^{\alpha\beta}}{r^2}\left[ 
\frac{\partial }{\partial x^3}(G_{\beta \nu}) \right]W^{3\nu}
\end{equation*}

\begin{equation*}
    = -\frac{g^{\alpha\beta}}{r^2}\left[ 
\frac{\partial }{\partial r}(r^2g_{\beta \nu}) \right]W^{3\nu}
\end{equation*}

\begin{equation*}
    = -2\frac{g^{\alpha\beta}}{r^2}rg_{\beta \nu}W^{3\nu}
\end{equation*}

\begin{equation*}
    = -\frac{2}{r}\delta^\alpha_\nu W^{3\nu}
\end{equation*}

\begin{equation*}
    = -\frac{2}{r}W^{3\alpha}
\end{equation*}

\begin{equation*}
    = -\frac{2}{r}\sqrt{G}\left[\rho \tilde{v}^\alpha V^3 - \frac{1}{\mu}\tilde{b}^\alpha B^3 + G^{\alpha 3}(P + \frac{|\pmb{B}|^2}{2\mu})\right]
\end{equation*}

\begin{equation*}
    = -\frac{2}{r}r^2\sqrt{g}\left[\rho \frac{1}{r}v^\alpha V^3 - \frac{1}{\mu}\frac{1}{r}b^\alpha B^3\right]
\end{equation*}

\begin{equation}
    = -2\sqrt{g}\left[\rho v^\alpha V^3 - \frac{1}{\mu}b^\alpha B^3\right]
\end{equation}

Combining these results, the relation for the surface divergence of the momentum flux is:

\begin{equation}
    w^{\alpha\beta}_{||\beta} = -\frac{\sqrt{g}}{\mu}b^\alpha(b^\beta_{||\beta} + 2 B^3) - 3\sqrt{g}\left[\rho v^\alpha V^3 - \frac{1}{\mu}b^\alpha B^3\right]
\end{equation}

or

\begin{equation}
    w^{\alpha\beta}_{||\beta} + 3\sqrt{g}\left[\rho v^\alpha V^3 - \frac{1}{\mu}b^\alpha B^3\right] = -\frac{\sqrt{g}}{\mu}b^\alpha(b^\beta_{||\beta} + 2 B^3)
\end{equation}

and with:

\begin{equation}
    w^{\alpha\beta} = \sqrt{g}\left(\rho v^\alpha v^\beta - \frac{1}{\mu}b^\alpha b^\beta + g^{\alpha\beta}\left[P + \frac{|\pmb{B}|^2}{2\mu}\right]\right)
\end{equation}

The final result is equation \eqref{momentum}.

\subsection{Third momentum equation}

The projection of the momentum equation in the last subsection reduced the number of equations for the momentum components from 3 down to 2. In order for the system to still be complete, another equation must be derived. This is derived from the equation for the third component of the momentum equation. Before being projected, that is:

\begin{equation}
    W\indices{^{3j}_{|j}} = -\frac{\sqrt{G}}{\mu}B^3B^j_{|j} 
\end{equation}

The RHS can be rewritten as:

\begin{equation*}
    -\frac{\sqrt{G}}{\mu}B^3B^j_{|j} = -\frac{r^2\sqrt{g}}{\mu}B^3\frac{b^\beta_{||\beta} + 2 B^3}{r}
\end{equation*}

\begin{equation}
    = -\frac{r\sqrt{g}}{\mu}B^3(b^\beta_{||\beta} + 2 B^3)
\end{equation}

The LHS is given by:

\begin{equation*}
    W\indices{^{3j}_{|j}} = \frac{\partial W^{3j}}{\partial x^j} + \overset{(G)}{\Gamma}\indices{_h^3_k}W^{hk}
\end{equation*}

The derivatives on the first term on the RHS can be separated into surface and radial derivatives:

\begin{equation}\label{Mom3ProjA}
    = \frac{\partial }{\partial \xi^\alpha}\left(\sqrt{G}\left[\rho \tilde{v}^\alpha V^3 - \frac{1}{\mu}\tilde{b}^\alpha B^3 + G^{\alpha 3}(P + \frac{|\pmb{B}|^2}{2\mu})\right]\right) + \frac{\partial W^{33}}{\partial x^3} + \overset{(G)}{\Gamma}\indices{_h^3_k}W^{hk}
\end{equation}

Each of the 3 terms in equation \eqref{Mom3ProjA} are treated individually for clarity. The first term is:

\begin{equation*}
    = \frac{\partial }{\partial \xi^\alpha}\left(r^2\sqrt{g}\left[\rho \frac{1}{r}v^\alpha V^3 - \frac{1}{\mu}\frac{1}{r}b^\alpha B^3 \right]\right)
\end{equation*}

\begin{equation}
    = r\frac{\partial }{\partial \xi^\alpha}\left(\sqrt{g}\left[\rho v^\alpha V^3 - \frac{1}{\mu}b^\alpha B^3 \right]\right)
\end{equation}

The second term in equation \eqref{Mom3ProjA} is:

\begin{equation*}
    \frac{\partial W^{33}}{\partial x^3} = \frac{\partial }{\partial x^3}\left(\sqrt{G}\left[\rho (V^3)^2 - \frac{1}{\mu}(B^3)^2 + G^{33}(P + \frac{|\pmb{B}|^2}{2\mu})\right]\right)
\end{equation*}

\begin{equation*}
    = \frac{\partial }{\partial r}\left(r^2\sqrt{g}\left[\rho (V^3)^2 - \frac{1}{\mu}(B^3)^2 + (P + \frac{|\pmb{B}|^2}{2\mu})\right]\right)
\end{equation*}

\begin{equation}
    = 2r\left(\sqrt{g}\left[\rho (V^3)^2 - \frac{1}{\mu}(B^3)^2 + (P + \frac{|\pmb{B}|^2}{2\mu})\right]\right)
\end{equation}

The third term in equation \eqref{Mom3ProjA} is:

\begin{equation*}
    \overset{(G)}{\Gamma}\indices{_h^3_k}W^{hk} = \frac{G^{3l}}{2}\left[ 
\frac{\partial G_{lh}}{\partial x^k} + \frac{\partial G_{lk}}{\partial x^h} - \frac{\partial G_{hk}}{\partial x^l} \right]W^{hk}
\end{equation*}

\begin{equation*}
    = \frac{1}{2}\left[ 
\frac{\partial G_{3h}}{\partial x^k} + \frac{\partial G_{3k}}{\partial x^h} - \frac{\partial G_{hk}}{\partial x^3} \right]W^{hk}
\end{equation*}

The first two terms in brackets are identically zero because of the form of the metric, which leaves:

\begin{equation*}
    = \frac{1}{2}\left[ - \frac{\partial G_{hk}}{\partial x^3} \right]W^{hk}
\end{equation*}

The bracketed term is zero when either $h=3$ or $k=3$, thus:

\begin{equation*}
    = \frac{1}{2}\left[ - \frac{\partial }{\partial r}(r^2g_{\gamma\nu}) \right]W^{\gamma\nu}
\end{equation*}

\begin{equation*}
    = -rg_{\gamma\nu}W^{\gamma\nu}
\end{equation*}

\begin{equation*}
    = -rg_{\gamma\nu}\sqrt{G}\left[\rho \tilde{v}^\gamma \tilde{v}^\nu - \frac{1}{\mu}\tilde{b}^\gamma\tilde{b}^\nu + \tilde{g}^{\gamma\nu}(P + \frac{|\pmb{B}|^2}{2\mu})\right]
\end{equation*}

\begin{equation*}
    = -rg_{\gamma\nu}\sqrt{g}\left[\rho v^\gamma v^\nu - \frac{1}{\mu}b^\gamma b^\nu + g^{\gamma\nu}(P + \frac{|\pmb{B}|^2}{2\mu})\right]
\end{equation*}

\begin{equation}
    = -r\sqrt{g}\left[\rho v_c^2 - \frac{1}{\mu}b_c^2 + 2(P + \frac{|\pmb{B}|^2}{2\mu})\right]
\end{equation}

Putting everything back together gives:

\begin{multline}
    W^{3j}_j = r\frac{\partial }{\partial \xi^\alpha}\left(\sqrt{g}\left[\rho v^\alpha V^3 - \frac{1}{\mu}b^\alpha B^3 \right]\right) \\ + 2r\left(\sqrt{g}\left[\rho (V^3)^2 - \frac{1}{\mu}(B^3)^2 + (P + \frac{|\pmb{B}|^2}{2\mu})\right]\right) \\ - r\sqrt{g}\left[\rho v_c^2 - \frac{1}{\mu}b_c^2 + 2(P + \frac{|\pmb{B}|^2}{2\mu})\right]
\end{multline}

\begin{equation}
    = r\left[\frac{\partial }{\partial \xi^\alpha}\left(\sqrt{g}\left[\rho v^\alpha V^3 - \frac{1}{\mu}b^\alpha B^3 \right]\right) + 2\sqrt{g}\left(\rho (V^3)^2 - \frac{1}{\mu}(B^3)^2\right) - \sqrt{g}\rho v_c^2 - \frac{\sqrt{g}}{\mu}b_c^2\right]
\end{equation}

Combining this with the RHS and cancelling the $r$'s leaves equation \eqref{mom3}.

\subsection{Magnetic equation}

The last equation to project is that for the components of the magnetic field, equation \eqref{3mag}. The LHS is the contracted covariant derivative (divergence) of a rank 2 tensor, or a rank 2 relative tensor of weight 0 as defined by \cite{Lovelock}. The contracted covariant derivative of which is given by the following:

Let $W\indices{^{ij}}$ be a rank 2 contravariant tensor, such as:

\begin{equation*}
    V^jB^i-V^iB^j
\end{equation*}

Then the contracted covariant derivative is given by:

\begin{equation}
W\indices{^{ij}_{|j}} = \frac{\partial W^{ij}}{\partial x^j} + \overset{(G)}{\Gamma}\indices{_h^i_k}W^{hk} + \overset{(G)}{\Gamma}\indices{_h^k_k}W^{hi}
\end{equation}

And likewise for a surface tensor of the same type and weight and its rescaled components.

\begin{equation}
\tilde{w}\indices{^{\alpha\beta}_{||\beta}} = \frac{\partial \tilde{w}^{\alpha\beta}}{\partial \xi^\beta} + \overset{(\tilde{g})}{\Gamma}\indices{_\gamma^\alpha_\nu}\tilde{w}^{\gamma\nu} + \overset{(\tilde{g})}{\Gamma}\indices{_\gamma^\nu_\nu}\tilde{w}^{\gamma\alpha}
\end{equation}

\begin{equation}
w\indices{^{\alpha\beta}_{||\beta}} = \frac{\partial w^{\alpha\beta}}{\partial \xi^\beta} + \overset{(g)}{\Gamma}\indices{_\gamma^\alpha_\nu}w^{\gamma\nu} + \overset{(g)}{\Gamma}\indices{_\gamma^\nu_\nu}w^{\gamma\alpha}
\end{equation}

Where the Christoffel symbols are defined in terms of the respective metric tensors, and:

\begin{equation}
    \tilde{w}^{\alpha\beta} = \tilde{v}^\beta\tilde{b}^\alpha-\tilde{v}^\alpha\tilde{b}^\beta
\end{equation}

and

\begin{equation}
    w^{\alpha\beta} = v^\beta b^\alpha-v^\alpha b^\beta
\end{equation}

We then have:

\begin{multline*}
    w\indices{^{\alpha\beta}_{||\beta}} = \frac{\partial }{\partial \xi^\beta}(r^2F^\alpha_iF^\beta_jW^{ij}) \\ + F^h_\gamma F^\alpha_k F^i_\nu \overset{(G)}{\Gamma}\indices{_h^k_i}r^2F^\gamma_l F^\nu_m W^{lm} \\ + F^h_\gamma F^\nu_k F^i_\nu \overset{(G)}{\Gamma}\indices{_h^k_i}r^2F^\gamma_l F^\alpha_m W^{lm}
\end{multline*}

\begin{multline*}9
    = r^2F^\alpha_iF^\beta_jF^n_\beta\frac{\partial W^{ij}}{\partial x^n} \\ + r^2F^h_\gamma F^\alpha_k F^i_\nu F^\gamma_l F^\nu_m \overset{(G)}{\Gamma}\indices{_h^k_i} W^{lm} \\ + r^2 F^h_\gamma F^\nu_k F^i_\nu F^\gamma_l F^\alpha_m \overset{(G)}{\Gamma}\indices{_h^k_i} W^{lm}
\end{multline*}

\begin{multline*}
    = r^2\biggl[ F^\alpha_i(\delta^n_j-N^n_j)\frac{\partial W^{ij}}{\partial x^n} \\ + F^\alpha_k (\delta^i_m-N^i_m)(\delta^h_l-N^h_l) \overset{(G)}{\Gamma}\indices{_h^k_i} W^{lm} \\ + F^\alpha_m(\delta^i_k-N^i_k)(\delta^h_l-N^h_l) \overset{(G)}{\Gamma}\indices{_h^k_i} W^{lm} \biggr]
\end{multline*}

The remaining projection factors can be pulled out after a suitable change of indices.

\begin{multline*}
    = r^2F^\alpha_i\biggl[ (\delta^n_j-N^n_j)\frac{\partial W^{ij}}{\partial x^n} \\ + (\delta^k_m-N^k_m)(\delta^h_l-N^h_l) \overset{(G)}{\Gamma}\indices{_h^i_k} W^{lm} \\ + (\delta^m_k-N^m_k)(\delta^h_l-N^h_l) \overset{(G)}{\Gamma}\indices{_h^k_m} W^{li} \biggr]
\end{multline*}

\begin{multline*}
    = r^2F^\alpha_i\biggl[ \frac{\partial W^{ij}}{\partial x^j} - \frac{\partial W^{i3}}{\partial x^3} \\ + (\delta^k_m\delta^h_l - \delta^k_m N^h_l - N^k_m\delta^h_l + N^k_m N^h_l) \overset{(G)}{\Gamma}\indices{_h^i_k} W^{lm} \\ + (\delta^m_k\delta^h_l - \delta^m_k N^h_l - N^m_k\delta^h_l + N^m_kN^h_l) \overset{(G)}{\Gamma}\indices{_h^k_m} W^{li} \biggr]
\end{multline*}

\begin{multline*}
    = r^2F^\alpha_i\biggl[ \frac{\partial W^{ij}}{\partial x^j} - \frac{\partial W^{i3}}{\partial x^3} \\ + \overset{(G)}{\Gamma}\indices{_h^i_k} W^{hk} - \overset{(G)}{\Gamma}\indices{_3^i_k} W^{3k} - \overset{(G)}{\Gamma}\indices{_h^i_3} W^{h3} + \overset{(G)}{\Gamma}\indices{_3^i_3} W^{33} \\ + \overset{(G)}{\Gamma}\indices{_h^k_k} W^{hi} - \overset{(G)}{\Gamma}\indices{_3^k_k} W^{3i} - \overset{(G)}{\Gamma}\indices{_h^3_3} W^{hi} + \overset{(G)}{\Gamma}\indices{_3^3_3} W^{3i} \biggr]
\end{multline*}

Noting that $W^{ij}=-W^{ji}$, we have:

\begin{equation*}
    = r^2F^\alpha_i\biggl[ W^{ij}_{|j} - \frac{\partial W^{i3}}{\partial x^3} + \overset{(G)}{\Gamma}\indices{_3^i_3} W^{33} - \overset{(G)}{\Gamma}\indices{_3^k_k} W^{3i} - \overset{(G)}{\Gamma}\indices{_h^3_3} W^{hi} + \overset{(G)}{\Gamma}\indices{_3^3_3} W^{3i} \biggr]
\end{equation*}

The third, fifth and sixth terms are zero, leaving:

\begin{equation}\label{ProjMagA}
    = r^2F^\alpha_i\biggl[ W^{ij}_{|j} - \frac{\partial W^{i3}}{\partial x^3} - \overset{(G)}{\Gamma}\indices{_3^k_k} W^{3i} \biggr]
\end{equation}

Each term is now treated individually. The first term is:

\begin{equation*}
    r^2F^\alpha_i W^{ij}_{|j} = -r^2F^\alpha_iV^iB^j_{|j}
\end{equation*}

\begin{equation*}
    = -r^2\tilde{v}^\alpha \frac{ b^\beta_{||\beta} + 2 B^3 }{r}
\end{equation*}

\begin{equation}
    = -v^\alpha (b^\beta_{||\beta} + 2 B^3)
\end{equation}

The second term in equation \eqref{ProjMagA} is:

\begin{equation*}
    - r^2F^\alpha_i\frac{\partial W^{i3}}{\partial x^3} = - r^2F^\alpha_i\frac{\partial }{\partial x^3}(V^3B^i-V^iB^3)
\end{equation*}

\begin{equation*}
    = - r^2\frac{\partial }{\partial r}\left[\frac{1}{r}(V^3b^\alpha-v^\alpha B^3)\right]
\end{equation*}

\begin{equation*}
    = r^2\frac{1}{r^2}(V^3b^\alpha-v^\alpha B^3)
\end{equation*}

\begin{equation}
    = V^3b^\alpha-v^\alpha B^3
\end{equation}

The third term in equation \eqref{ProjMagA} is:

\begin{equation*}
    -r^2F^\alpha_i\overset{(G)}{\Gamma}\indices{_3^k_k} W^{3i} = -r^2F^\alpha_i \frac{1}{\sqrt{G}}\frac{\partial \sqrt{G}}{\partial x^3} W^{3i}
\end{equation*}

\begin{equation*}
    = -r^2 \frac{1}{r^2\sqrt{g}}\frac{\partial r^2\sqrt{g}}{\partial r} \frac{1}{r}(V^3b^\alpha-v^\alpha B^3)
\end{equation*}

\begin{equation*}
    = - 2r\frac{1}{\sqrt{g}} \sqrt{g} \frac{1}{r}(V^3b^\alpha-v^\alpha B^3)
\end{equation*}

\begin{equation}
    = - 2(V^3b^\alpha-v^\alpha B^3)
\end{equation}

Putting it all together:

\begin{equation}
    w^{\alpha\beta}_{||\beta} = -v^\alpha (b^\beta_{||\beta} + 2 B^3) + V^3b^\alpha-v^\alpha B^3 - 2(V^3b^\alpha-v^\alpha B^3)
\end{equation}

Thus:

\begin{equation}
    (v^\beta b^\alpha - v^\alpha b^\beta)_{||\beta}  + V^3b^\alpha-v^\alpha B^3 = -v^\alpha (b^\beta_{||\beta} + 2 B^3)
\end{equation}

which is equation \eqref{magnet}.

\subsection{Third magnetic equation}

The above projection has again reduced the number of the equations from 3 down to 2. Thus another equation must be derived to close the system. This again comes from looking at the equation for the third magnetic field component, which is:

\begin{equation}
    W^{3j}_{|j} = -V^3B^j_{|j}
\end{equation}

where:

\begin{equation}
    W^{3j}=V^jB^3-V^3B^j
\end{equation}

The RHS is:

\begin{equation}
    -V^3B^j_{|j} = -V^3\frac{b^\beta_{||\beta} + 2 B^3}{r}
\end{equation}

The LHS is:

\begin{equation*}
    W^{3j}_{|j} = \frac{\partial }{\partial x^j}W^{3j} + \overset{(G)}{\Gamma}\indices{_m^3_j} W^{mj} + \overset{(G)}{\Gamma}\indices{_m^j_j} W^{m3}
\end{equation*}

separating the derivatives:

\begin{multline*}
    = \frac{\partial }{\partial \xi^\beta}(\tilde{v}^\beta B^3-V^3\tilde{b}^\beta) + \frac{\partial }{\partial r}(V^3B^3-V^3B^3) \\ + \frac{G^{3h}}{2}\left[ 
\frac{\partial G_{hm}}{\partial x^j} + \frac{\partial G_{hj}}{\partial x^m} - \frac{\partial G_{mj}}{\partial x^h} \right] W^{mj} \\ + \frac{1}{\sqrt{G}}\frac{\partial \sqrt{G}}{\partial x^m} W^{m3}
\end{multline*}

\begin{equation*}
    = \frac{1}{r}\frac{\partial }{\partial \xi^\beta}(v^\beta B^3-V^3b^\beta) + \frac{1}{2}\left[ 
\frac{\partial G_{3m}}{\partial x^j} + \frac{\partial G_{3j}}{\partial x^m} - \frac{\partial G_{mj}}{\partial x^3} \right] W^{mj} + \frac{1}{r^2\sqrt{g}}\frac{\partial r^2\sqrt{g}}{\partial x^m} W^{m3}
\end{equation*}

\begin{multline*}
    = \frac{1}{r}\frac{\partial }{\partial \xi^\beta}(v^\beta B^3-V^3b^\beta) + \frac{1}{2}\left[ - \frac{\partial G_{mj}}{\partial x^3} \right] W^{mj} \\ + \frac{1}{r^2\sqrt{g}}\frac{\partial r^2\sqrt{g}}{\partial \xi^\beta} (\tilde{v}^\beta B^3-V^3\tilde{b}^\beta) + \frac{1}{r^2\sqrt{g}}\frac{\partial r^2\sqrt{g}}{\partial x^3} W^{33}
\end{multline*}

\begin{equation*}
    = \frac{1}{r}\frac{\partial }{\partial \xi^\beta}(v^\beta B^3-V^3b^\beta) - \frac{1}{2}\frac{\partial }{\partial r}(r^2g_{\gamma\nu}) \tilde{w}^{\gamma\nu} + \frac{1}{\sqrt{g}}\frac{\partial \sqrt{g}}{\partial \xi^\beta} (\tilde{v}^\beta B^3-V^3\tilde{b}^\beta)
\end{equation*}

\begin{equation*}
    = \frac{1}{r}\frac{\partial }{\partial \xi^\beta}(v^\beta B^3-V^3b^\beta) - g_{\gamma\nu}\tilde{w}^{\gamma\nu} + \frac{1}{\sqrt{g}}\frac{\partial \sqrt{g}}{\partial \xi^\beta} \frac{1}{r}(v^\beta B^3-V^3b^\beta)
\end{equation*}

By the symmetry of the metric and the anti symmetry of the magnetic field flux, the second term is zero. This leaves:

\begin{equation}
    = \frac{1}{r}\left[\frac{\partial }{\partial \xi^\beta}(v^\beta B^3-V^3b^\beta) + (v^\beta B^3-V^3b^\beta)\frac{1}{\sqrt{g}}\frac{\partial \sqrt{g}}{\partial \xi^\beta} \right]
\end{equation}

Putting everything back together gives:

\begin{equation}
    \frac{1}{r}\left[\frac{\partial }{\partial \xi^\beta}(v^\beta B^3-V^3b^\beta) + (v^\beta B^3-V^3b^\beta)\frac{1}{\sqrt{g}}\frac{\partial \sqrt{g}}{\partial \xi^\beta} \right] = -V^3\frac{b^\beta_{||\beta} + 2 B^3}{r}
\end{equation}

and thus:

\begin{equation}
    \frac{\partial }{\partial \xi^\beta}(v^\beta B^3-V^3b^\beta) + (v^\beta B^3-V^3b^\beta)\frac{1}{\sqrt{g}}\frac{\partial \sqrt{g}}{\partial \xi^\beta} = -V^3(b^\beta_{||\beta} + 2 B^3)
\end{equation}

which is equation \eqref{mag3}. The full system of Ideal MHD equations is now projected onto the unit sphere.

\section{Elliptic-Hyperbolic property}\label{hyperbolic}

A general 1st-order system of $m$ differential equations in $n$ spatial dimensions has the following form:

\begin{equation}\label{general}
U_t + \sum_{i=1}^{n}\bar{A}^iU_{x^i}+\bar{S}=0
\end{equation}

Where $U:\mathbb{R}^n\rightarrow\mathbb{R}^m$ is a column vector of the dependent variables and each $\bar{A}^i$ is an $m$ by $m$ matrix that can in general depend on $U$ and $\pmb{x}$. S is a column vector of source terms. 

\begin{definition}
A system of the form \eqref{general} is said to be strictly hyperbolic if $\forall \pmb{w}\in \mathbb{R}^n, |\pmb{w}|=1$, the eigenvalues of $\bar{A}_w=\sum_{i=1}^{n}w_i\bar{A}^i$ are real and distinct. If they are all real, but not all distinct, the system is non-strictly hyperbolic. If any of the eigenvalues are complex then the system is said to be elliptic \cite{CHEN20051,evansPDE,lax2006hyperbolic,renardy}.
\end{definition}


Often a first order system will come in the form:

\begin{equation}\label{fullform}
A_0U_t + \sum_{i=1}^{n}A^iU_{x^i}+S=0
\end{equation}

with a matrix multiplying the time derivative term. For the type of this system to be determined, it must first be put into the form \eqref{general}. This is done by multiplying by the inverse of the leading matrix, $A_0$, and defining $\bar{A}^i=A_0^{-1}A^i$. 

Two useful results regarding the type of system \eqref{fullform} are now presented.

\begin{theorem}[Invariance under matrix multiplication]
    The type of system \eqref{fullform} is unchanged under multiplication by an invertible matrix. That is that $\forall M\in\mathbb{R}^{m\times m}$, such that $M^{-1}$ exists, the characteristic speeds of system \eqref{fullform} are the same as for the system:
    \begin{equation}
        (MA_0)U_t + \sum_{i=1}^{n}(MA^i)U_{x^i}+S=0
    \end{equation}
\end{theorem}

A proof of this is simple and thus left out.

The second result was proved by Evans in \cite{evansPDE} in a slightly different format than presented here.

\begin{theorem}[Invariance under change of dependent variables]\label{chngOfVars}
    Let $\Phi:\mathbb{R}^n\rightarrow\mathbb{R}^n$ be a smooth diffeomorphism with invertible Jacobian matrix $D\Phi$ and inverse map $\Psi$. Let $U$ be as in system \eqref{fullform}. Then $\Phi(U)$ satisfies the system:
    \begin{equation}
        C_0\Phi_t + \sum_{i=1}^{n}C^i\Phi_{x^i}+S^*=0
    \end{equation}
    which has the same characteristic speeds as system \eqref{fullform}.
\end{theorem}

\begin{proof}[Proof of \ref{chngOfVars}]
    We have that $U$ satisfies the relationship \eqref{fullform}. Furthermore:
    \begin{equation*}
        A_0U_t + \sum_{i=1}^{n}A^iU_{x^i}+S = A_0(D\Phi)^{-1}D\Phi U_t + \sum_{i=1}^{n}A^i(D\Phi)^{-1}D\Phi U_{x^i}+S
    \end{equation*}
    
    \begin{equation*}
        = A_0(D\Phi)^{-1}\Phi_t(U) + \sum_{i=1}^{n}A^i(D\Phi)^{-1}\Phi_{x^i}(U)+S
    \end{equation*}
    
    \begin{equation*}
        = A_0(\Psi(\Phi))(D\Phi)^{-1}\Phi_t(U) + \sum_{i=1}^{n}A^i(\Psi(\Phi))(D\Phi)^{-1}\Phi_{x^i}(U)+S(\Psi(\Phi))
    \end{equation*}
    
    \begin{equation*}
        = C_0\Phi_t + \sum_{i=1}^{n}C^i\Phi_{x^i}(U)+S^*
    \end{equation*}
    
    where $C_0=A_0(\Psi(\Phi))(D\Phi)^{-1}$, $C^i=A^i(\Psi(\Phi))(D\Phi)^{-1}$, and $S^*=S(\Psi(\Phi))$. To put this in the form \eqref{general}, we multiply by the inverse of $C_0$, which is $C_0^{-1}=D\Phi A_0^{-1}$. The resulting system is:
    
    \begin{equation*}
        \Phi_t + \sum_{i=1}^{n}\bar{C}^i\Phi_{x^i}(U)+\bar{S}^*=0
    \end{equation*}
    
    where $\bar{C}^i=C_0^{-1}C^i=D\Phi A_0^{-1}A^i(D\Phi)^{-1}$, and $\bar{S}^*=D\Phi A_0^{-1}S^*$. Finally, $\forall\pmb{w}\in \mathbb{R}^n, |\pmb{w}|=1$ we have:
    
    \begin{equation*}
        \bar{C}_w = \sum_{i=1}^{n}w_i\bar{C}^i = \sum_{i=1}^{n}w_iD\Phi A_0^{-1}A^i(D\Phi)^{-1} = D\Phi\left(\sum_{i=1}^{n}w_i A_0^{-1}A^i\right)(D\Phi)^{-1}
    \end{equation*}
    
    This matrix is similar to $\bar{A}_w = \sum_{i=1}^{n}w_i A_0^{-1}A^i$ and thus has the same spectrum. Therefore, the characteristic speeds and type of this system are the same as those of system \eqref{fullform}.
\end{proof}

The above two theorems will be used in the next section to aid the analysis of the characteristic structure of system \eqref{TheEq}.

\section{Eigenvalues}\label{valsvects}

As presented, system \eqref{TheEq} most naturally has the dependent variables

\begin{equation}
    U = \left[ \begin{matrix} \rho & v^1 & v^2 & V^3 & e & b^1 & b^2 & B^3 \end{matrix} \right]^T
\end{equation}

After using the product rule and/or chain rule to expand all the derivatives until they are in terms of derivatives of these individual dependent variables, the system has the quasilinear form:

\begin{equation}\label{quasilinear1}
\sum_{\alpha=1}^{2}A^\alpha U_{\xi^\alpha}+S=0
\end{equation}

Because the analytical calculation of eigenvalues of an 8 by 8 matrix can be tedious, we are motivated to manipulate the matrices into a more manageable form. To this end we switch from using $e$ as a dependent variable to using $P$. Clearly the map:

\begin{equation}
    \Phi:\left[ \begin{matrix} \rho & v^1 & v^2 & V^3 & e & b^1 & b^2 & B^3 \end{matrix} \right]^T \rightarrow \left[ \begin{matrix} \rho & v^1 & v^2 & V^3 & P(\rho,e) & b^1 & b^2 & B^3 \end{matrix} \right]^T
\end{equation}

satisfies the conditions of Theorem \ref{chngOfVars} and so maintains the eigenstructure of the system. We then take the liberty of multiplying system \eqref{quasilinear1} by the matrix $D\Phi M^{-1}$ where:

\begin{equation}
    D\Phi = \left[
\begin{matrix}
   1 & 0 & 0 & 0 & 0 & 0 & 0 & 0 \\
   0 & 1 & 0 & 0 & 0 & 0 & 0 & 0 \\
   0 & 0 & 1 & 0 & 0 & 0 & 0 & 0 \\
   0 & 0 & 0 & 1 & 0 & 0 & 0 & 0 \\
   P_\rho & 0 & 0 & 0 & P_e & 0 & 0 & 0\\
   0 & 0 & 0 & 0 & 0 & 1 & 0 & 0\\
   0 & 0 & 0 & 0 & 0 & 0 & 1 & 0 \\
   0 & 0 & 0 & 0 & 0 & 0 & 0 & 1 \\
\end{matrix}
\right]
\end{equation}

and

\begin{equation}
    M = \left[
\begin{smallmatrix}
   \sqrt{g} & 0 & 0 & 0 & 0 & 0 & 0 & 0 \\
   \sqrt{g}v^1 & \sqrt{g}\rho & 0 & 0 & 0 & 0 & 0 & 0 \\
   \sqrt{g}v^2 & 0 & \sqrt{g}\rho & 0 & 0 & 0 & 0 & 0 \\
   \sqrt{g}V^3 & 0 & 0 & \sqrt{g}\rho & 0 & 0 & 0 & 0 \\
   \sqrt{g}E & \sqrt{g}\rho(g_{1\alpha}v^\alpha) & \sqrt{g}\rho(g_{2\alpha}v^\alpha) & \sqrt{g}\rho V^3 & \sqrt{g}\rho & \sqrt{g}(g_{1\alpha}b^\alpha)/\mu & \sqrt{g}(g_{2\alpha}b^\alpha)/\mu & \sqrt{g}B^3/\mu\\
   0 & 0 & 0 & 0 & 0 & 1 & 0 & 0\\
   0 & 0 & 0 & 0 & 0 & 0 & 1 & 0 \\
   0 & 0 & 0 & 0 & 0 & 0 & 0 & 1 \\
\end{smallmatrix}
\right]
\end{equation}

This puts system \eqref{TheEq} in the new quasilinear form:

\begin{equation}\label{quasilinear2}
\sum_{\alpha=1}^{2}C^\alpha \Phi_{\xi^\alpha}+S^*=0
\end{equation}

with dependent variables:

\begin{equation}
\Phi = \left[ \begin{matrix} \rho & v^1 & v^2 & V^3 & P & b^1 & b^2 & B^3 \end{matrix} \right]^T
\end{equation}

and matrices:

\begin{equation}
    C^1 = \left[
\begin{matrix}
   v^1 & \rho & 0 & 0 & 0 & 0 & 0 & 0 \\
   0 & v^1 & 0 & 0 & g^{11}/\rho & -g^{12}(g_{2\alpha}b^{\alpha})/(\mu\rho) & g^{11}(g_{2\alpha}b^{\alpha})/(\mu\rho) & g^{11}B^3/(\mu\rho) \\
   0 & 0 & v^1 & 0 & g^{12}/\rho & g^{12}(g_{1\alpha}b^{\alpha})/(\mu\rho) & -g^{11}(g_{1\alpha}b^{\alpha})/(\mu\rho) & g^{12}B^3/(\mu\rho) \\
   0 & 0 & 0 & v^1 & 0 & 0 & 0 & -b^1/(\mu\rho) \\
   0 & c^2\rho & 0 & 0 & v^1 & 0 & 0 & 0\\
   0 & 0 & 0 & 0 & 0 & v^1 & 0 & 0\\
   0 & b^2 & -b^1 & 0 & 0 & 0 & v^1 & 0 \\
   0 & B^3 & 0 & -b^1 & 0 & 0 & 0 & v^1 \\
\end{matrix}
\right]
\end{equation}

and

\begin{equation}
    C^2 = \left[
\begin{matrix}
   v^2 & 0 & \rho & 0 & 0 & 0 & 0 & 0 \\
   0 & v^2 & 0 & 0 & g^{21}/\rho & -g^{22}(g_{2\alpha}b^{\alpha})/(\mu\rho) & g^{12}(g_{2\alpha}b^{\alpha})/(\mu\rho) & g^{21}B^3/(\mu\rho) \\
   0 & 0 & v^2 & 0 & g^{22}/\rho & g^{22}(g_{1\alpha}b^{\alpha})/(\mu\rho) & -g^{12}(g_{1\alpha}b^{\alpha})/(\mu\rho) & g^{22}B^3/(\mu\rho) \\
   0 & 0 & 0 & v^2 & 0 & 0 & 0 & -b^2/(\mu\rho) \\
   0 & 0 & c^2\rho & 0 & v^2 & 0 & 0 & 0\\
   0 & -b^2 & b^1 & 0 & 0 & v^2 & 0 & 0 \\
   0 & 0 & 0 & 0 & 0 & 0 & v^2 & 0\\
   0 & 0 & B^3 & -b^2 & 0 & 0 & 0 & v^2 \\
\end{matrix}
\right]
\end{equation}

For Hyperbolicity to be assessed, a spatial variable must be chosen to be treated as time-like. Without loss of generality, $\xi^1$ is chosen. System \eqref{quasilinear2} is then multiplied by the inverse of $B^1$ giving:

\begin{equation}
\Phi_{\xi^1}+\bar{C}\Phi_{\xi^2}+\bar{S}^*=0
\end{equation}

where $\bar{C} = (C^1)^{-1}C^2$ and $\bar{S}^*=(C^1)^{-1}S^*$. There is only one matrix left, so we simply take $w=1$ and $\bar{C}_w = \bar{C}$. The characteristic polynomial of $\bar{C}$ was computed and factored using Wolfram Mathematica \cite{Mathematica} following the procedure described by \cite{Sterck}. This resulted in the following relationships for the eigenvalues.

The first two eigenvalues are both given by:

\begin{equation}
    \lambda=\frac{v^2}{v^1}
\end{equation}

The second two are:

\begin{equation}
    \lambda = \frac{b^2 \pm \sqrt{\mu\rho}v^2 }{b^1 \pm \sqrt{\mu\rho}v^1}
\end{equation}

And the last four eigenvalues satisfy the relationship:

\begin{multline}\label{last4evals2}
    (v^2-\lambda v^1)^2 =\\ 
    \frac{g}{2\mu\rho} \Biggl[ (g^{22} - 2g^{12}\lambda + g^{11}\lambda^2) \frac{(b_1)^2 (|\pmb{B}|^2+c^2\mu\rho)}{b_c^2-(b^2)^2} \pm \\
    \sqrt{ (g^{22} - 2g^{12}\lambda + g^{11}\lambda^2) \left( \frac{-4c^2(b^2-b^1\lambda)^2\mu\rho}{g^2} + (g^{22} - 2g^{12}\lambda + g^{11}\lambda^2)\left(\frac{(b_1)^2 (|\pmb{B}|^2+c^2\mu\rho)}{b_c^2-(b^2)^2}\right)^2 \right)  } \Biggr]   
\end{multline}

where the speed of sound, $c$, is given by:

\begin{equation}
    c = \frac{\sqrt{PP_e + \rho^2 P_\rho} }{\rho}
\end{equation}

\begin{remark}
We note that in the event that $P$ is a function of $\rho$ only this expression reduces to the expression:

\begin{equation}
    c = \sqrt{\frac{\partial P}{\partial \rho}}
\end{equation}

and for an ideal gas with $P=(\gamma-1)\rho e$ this gives:

\begin{equation}
    c = \sqrt{\frac{\gamma P}{\rho}}
\end{equation}
\end{remark}

The first four eigenvalues are all real, with the first two coinciding. Graphical and/or numerical methods can be used to demonstrate that equation \eqref{last4evals2} will sometimes have four real solutions, but sometimes will not. Depending on the dependent variables, it is possible that the eigenvalues will be complex. In general then, one should expect the system to change type within the domain, being either hyperbolic or elliptic.

\subsection{Pseudo-Time Dependency}

The time derivative terms in Equation \eqref{transient} are not compatible with the conical assumption because the $r$ dependency fails to disappear. However, it is often convenient to solve a steady problem numerically by marching in time until the solution stabilizes. For that purpose one could reinsert the time derivatives with the appropriate metrics and treat the problem as unsteady. 

For this nonphysical problem, the form is:

\begin{equation}
A_0U_t + \sum_{\alpha=1}^{2}A^\alpha U_{\xi^\alpha}=0
\end{equation}

And after multiplying by the inverse of $A_0$:

\begin{equation}
U_t + \sum_{\alpha=1}^{2}\bar{A}^\alpha U_{\xi^\alpha}=0
\end{equation}

defining $\bar{A}^\alpha = A_0^{-1}A^\alpha$. Since the system is given for the contravariant components of the velocity we take $\pmb{w}$ to be a covariant vector such that $g^{\alpha\beta}w_\alpha w_\beta = 1$, and form the linear combination $\bar{A}_w=\sum_{i=1}^{n}w_i\bar{A}^i$. The eigenvalues are:

\begin{equation}
    \lambda(\bar{A}_w) = \pmb{v}\cdot\pmb{w}, \pmb{v}\cdot\pmb{w}, \left(\pmb{v}\pm\frac{\pmb{b}}{\sqrt{\mu\rho}}\right)\cdot\pmb{w}, \pmb{v}\cdot\pmb{w}\pm c_f, \pmb{v}\cdot\pmb{w}\pm c_s
\end{equation}

where $c_f$ and $c_s$ are respectively the fast and slow magneto-acoustic waves which satisfy the relationships:

\begin{equation}
    c_f^2 = \frac{1}{2}\left( c^2 + \frac{|\pmb{B}|^2}{\mu\rho} + \sqrt{ \left(\frac{|\pmb{B}|^2}{\mu\rho} + c^2 \right)^2 - 4c^2\frac{(\pmb{b}\cdot\pmb{w})^2}{\mu\rho} } \right)
\end{equation}

\begin{equation}
    c_s^2 = \frac{1}{2}\left( c^2 + \frac{|\pmb{B}|^2}{\mu\rho} - \sqrt{ \left(\frac{|\pmb{B}|^2}{\mu\rho} + c^2 \right)^2 - 4c^2\frac{(\pmb{b}\cdot\pmb{w})^2}{\mu\rho} } \right)
\end{equation}

These characteristic speeds are the same as for general unsteady Ideal MHD, all of which are real, but at least two of them coincide. Therefore the system is everywhere non-strictly hyperbolic. This result is also analogous to the case of Cartesian unsteady Ideal MHD \cite{Sterck,JamesonMHD,ShangRecentResearch} which is to be expected.

\section{Conclusion}

We have thus systematically derived a system of equations which describe Ideal Magnetohydrodynamic flow past a cone of arbitrary cross section. The assumption of conical invariance on the 3D flow field allowed the equations to reduce to a system defined on the surface of the sphere. By using the machinery of tensor calculus, the resulting system is valid for any coordinate system defined on the surface of the sphere which is convenient for when the equations are solved numerically. Based on the characteristic analysis, it can be demonstrated that the system can be hyperbolic or elliptic depending on the solution. In general then, it can also change type within the domain.

\section*{Acknowledgement}

This research was supported in part by an appointment to the Student Research Participation Program at the U.S. Air Force Institute of Technology administered by the Oak Ridge Institute for Science and Education through an interagency agreement between the U.S. Department of Energy and USAFIT.

\bibliographystyle{amsplain}
\bibliography{references.bib}

\end{document}